\documentclass[12pt,letterpaper,reqno]{amsart}

\usepackage{graphicx}
\usepackage{amsfonts}
\usepackage{amssymb}
\usepackage{amsthm}
\usepackage{amsmath}

\newcommand\ts{\textsuperscript}
\newcommand\be{\begin{equation}}
\newcommand\ee{\end{equation}}
\newcommand\bea{\begin{eqnarray}}
\newcommand\eea{\end{eqnarray}}
\newcommand\bi{\begin{itemize}}
\newcommand\ei{\end{itemize}}
\newcommand\ben{\begin{enumerate}}
\newcommand\een{\end{enumerate}}
\newcommand\bc{\begin{center}}
\newcommand\ec{\end{center}}
\newcommand\ba{\begin{array}}
\newcommand\ea{\end{array}}

\newcommand{\R}{\ensuremath{\mathbb{R}}}

\newcommand{\N}{\mathbb{N}}

\newcommand{\E}{\ensuremath{\mathbb{E}}}
\newcommand{\p}{\ensuremath{\mathbb{P}}}

\newcommand{\f}{\mathcal{F}}

\newtheorem{thm}{Theorem}[section]
\newtheorem{conj}[thm]{Conjecture}
\newtheorem{cor}[thm]{Corollary}
\newtheorem{lem}[thm]{Lemma}

\newtheorem{defi}[thm]{Definition}

\newtheorem{rek}[thm]{Remark}

\newcommand{\ncr}[2]{{#1 \choose #2}}

\newcommand{\gep}{\epsilon}
\newcommand{\gl}{\lambda}
\newcommand{\ga}{\alpha}
\newcommand{\gb}{\beta}

\numberwithin{equation}{section}

\begin{document}

\title{Distribution of Eigenvalues of Highly Palindromic Toeplitz Matrices}

\author[Jackson]{Steven Jackson}\email{Steven.R.Jackson@williams.edu}
\address{Department of Mathematics and Statistics, Williams College, Williamstown, MA 01267}
\curraddr{Department of Physics, Princeton University, Princeton, NJ 08544}

\author[Miller]{Steven J. Miller}\email{Steven.J.Miller@williams.edu}
\address{Department of Mathematics and Statistics, Williams College, Williamstown, MA 01267}

\author[Pham]{Thuy Pham}\email{tvp1@williams.edu}
\address{Department of Mathematics and Statistics, Williams College, Williamstown, MA 01267}

\subjclass[2010]{15B52, 60F05, 11D45 (primary), 60F15, 60G57, 62E20 (secondary)}

\keywords{limiting spectral measure, Toeplitz matrices, random matrix theory, Diophantine equations, convergence, method of moments}

\date{\today}

\thanks{This work was done at the 2009 SMALL Undergraduate Research Project at Williams College, funded by NSF GRANT DMS-0850577 and Williams College; it is a pleasure to thank them and the other participants. The second named author was also partly supported by NSF grant DMS0600848 and DMS0970067.}

\begin{abstract}
Consider the ensemble of real symmetric Toeplitz matrices whose entries are i.i.d random variables chosen from a fixed probability distribution $p$ of mean 0, variance 1 and finite higher moments. Previous work \cite{BDJ,HM} showed that the limiting spectral measures (the density of normalized eigenvalues) converge in probability and almost surely to a universal distribution almost that of the Gaussian, independent of $p$. The deficit from the Gaussian distribution is due to obstructions to solutions of Diophantine equations and can be removed (see \cite{MMS}) by making the first row palindromic. In this paper, we study the case where there is more than one palindrome in the first row of a real symmetric Toeplitz matrix. Using the method of moments and an analysis of the resulting Diophantine equations, we show that the moments of this ensemble converge to an universal distribution with a fatter tail than any previously seen limiting spectral measure.
\end{abstract}

\maketitle

\tableofcontents


\section{Introduction}

\subsection{Background}

Since its inception, Random Matrix Theory has been a powerful tool in modeling highly complicated systems, with applications in statistics \cite{Wis}, nuclear physics \cite{Wig1, Wig2, Wig3, Wig4, Wig5} and number theory \cite{KS1, KS2, KeSn}; see \cite{FM} for a history of the development of some of these connections. An interesting problem in Random Matrix Theory is to study sub-ensembles of real symmetric matrices by introducing additional structure. One of those sub-ensembles is the family of real symmetric Toeplitz matrices; these matrices are constant along the diagonals: \be A_N \ = \ \left(\begin{array}{ccccc}
b_{0}  &  b_{1}  & b_{2}  & \cdots & b_{N-1} \\
b_{-1} &  b_{0}  & b_{1}  & \cdots & b_{N-2} \\
b_{-2} &  b_{-1} & b_{0}  & \cdots & b_{N-3} \\
\vdots & \vdots  & \vdots & \ddots & \vdots \\
b_{1-N} &  b_{2-N} & b_{3-N} & \cdots & b_0 \\
\end{array}\right), \ \ \ \ a_{ij} \ = \ b_{j-i}.
\ee

Initially numerical investigations suggested that the density of the normalized eigenvalues was given by the standard normal; however, Bose, Chatterjee, Gangopadhyay \cite{BCG},  Bryc, Dembo and Jiang \cite{BDJ} and Hammond and Miller \cite{HM} showed that this is not the case; in particular, the fourth moment is 2 2/3 and not 3. The analysis in \cite{HM} shows that although the moments grow more slowly than the Gaussian's, they grow sufficiently fast to determine a universal distribution with unbounded support. The deficit from the standard Gaussian's moments is due to obstructions to Diophantine equations.

In \cite{MMS}, Massey, Miller and Sinsheimer found that, by imposing additional structure on the Toeplitz matrices by making the first row a palindrome, the Diophantine obstructions vanish and the limiting spectral measure converges in probability and almost surely to the standard Gaussian.  A fascinating question to ask is how the behavior of the normalized eigenvalues changes if we impose other constraints. Basak and Bose \cite{BB1}, Kargin \cite{Kar} and Liu and Wang \cite{LW} obtain results for ensembles of Toeplitz (and other) matrices that are also band matrices, with the results depending on the relative size of the band length to the dimension of the matrices. Another direction is that of Basak and Bose \cite{BB2}, where each entry is scaled by the square root of the number of times that entry appears in the matrix. These ensembles are special cases of patterned matrices governed by a link function; see also \cite{BanBo,BHS}. In this paper we explore another generalization by studying the effect of increasing the palindromicity on the distribution of the eigenvalues. We begin by listing our notation below, and then stating our results in \S\ref{sec:results}.

\subsection{Notation}

\begin{defi}\label{def:descriptionhighlytoepensemble} For fixed $n$, we consider $N\times N$ real symmetric Toeplitz matrices in which the first row is $2^n$ copies of a palindrome. We always assume $N$ to be a multiple of  $2^n$ so that each element occurs exactly $2^{n+1}$ times in the first row. For instance, a doubly palindromic Toeplitz matrix (henceforth referred to as a DPT matrix) is of the form: \be\label{eq:defrsptmat} A_N\ =\ \left(
\begin{array}{cccccccccc}
b_0  &  b_1  &  \cdots  & b_1 & b_0 & b_0 & b_1 & \cdots  & b_1   & b_0  \\
b_1  &  b_0  &  \cdots  & b_2 & b_1 & b_0 & b_0 & \cdots  & b_2   & b_1  \\
b_2  &  b_1  &  \cdots  & b_3 & b_2 & b_1 & b_0 & \cdots  & b_3   & b_2  \\
\vdots  &  \vdots  &  \ddots  & \vdots & \vdots & \vdots & \vdots & \ddots  & \vdots   & \vdots  \\
b_2  &  b_3  &  \cdots  & b_0 & b_1 & b_2 & b_3 & \cdots  & b_1   & b_2  \\
b_1  &  b_2  &  \cdots  & b_0 & b_0 & b_1 & b_2 & \cdots  & b_0   & b_1  \\
b_0  &  b_1  &  \cdots  & b_1 & b_0 & b_0 & b_1 & \cdots  & b_1   & b_0  \\
\end{array}\right) \ee
We always assume the entries of our matrices are i.i.d.r.v. chosen from some distribution $p$ with mean 0, variance 1 and finite higher moments. The entries of the matrices are constant along diagonals. Furthermore, entries on two diagonals that are $N/2^n$ diagonals apart from each other are also equal. Finally, entries on two diagonals symmetric within a palindrome are also equal.
\end{defi}

To succinctly keep track of which elements are equal, we may introduce a \emph{link function} $\psi: \{1,\dots,N\}^2 \to \{1,\dots,N\}$ and new parameters $b_\ell$ such that $a_{ij} = b_{\psi (i,j)}$, where
\begin{equation}
\psi (i,j) \ = \
\begin{cases}
\ \ |i - j| \mod 2^n   &\ {\rm if}\  \ \ |i - j| \mod 2^n < N / 2^{n+1}  \\
- |i - j| \mod 2^n & \ {\rm if}\ \ \ |i - j| \mod 2^n  > N / 2^{n+1}. \\
\end{cases}
\end{equation}
Each $N \times N$ matrix $A_N$ in this ensemble can be identified with a vector in $\mathbb R^{N/2^n}$ by $A_N \leftrightarrow (b_0(A_N), b_1(A_N), \dots, b_{N/2^n}(A_N))$.


Before stating our results, we recall the various types of convergence. The following passage is paraphrased from \cite{MMS} with permission. Fix $n$ and for each integer $N$ let $\Omega_{n,N}$ denote the set of $N \times N$
real symmetric Toeplitz matrices whose first row is $2^n$ copies of a palindrome. We construct a
probability space $(\Omega_{n,N},\f_{n,N},\p_{n,N})$ by setting \bea & &
\p_{n,N}\left(\left\{A_N\in\Omega_{n,N}: b_{i}(A_N) \in [\ga_i, \gb_i]\
{\rm for}\ i \in \left\{0,\dots,N / 2^n - 1\right\}\right\}\right)
\nonumber\\ & & \ \ \ \ \ \ \ \ = \ \prod_{i=1}^M
\int_{x_i=\ga_i}^{\gb_i} p(x_i)dx_i, \eea where each $dx_i$ is
Lebesgue measure. To each $A_N\in \Omega_{n,N}$ we attach a spacing
measure by placing a point mass of size $1/N$ at each normalized
eigenvalue\footnote{From the eigenvalue trace lemma
($\text{Trace}(A_N^2) = \sum_i \gl_i^2(A_N)$) and the Central Limit
Theorem, we see that the eigenvalues of $A_N$ are of order
$\sqrt{N}$. This is because $\text{Trace}(A_N^2) = \sum_{i,j=1}^N
a_{ij}^2$, and since each $a_{ij}$ is drawn from a mean $0$,
variance $1$ distribution, $\text{Trace}(A_N^2)$ is of size $N^2$.
This suggests the appropriate scale for normalizing the eigenvalues
is to divide each by $\sqrt{N}$.} $\gl_i(A_N)$: \be\label{eq:normspacingmeasure}
\mu_{n, A_N}(x)dx \ = \ \frac{1}{N} \sum_{i=1}^N \delta\left( x -
\frac{\gl_i(A_N)}{\sqrt{N}} \right)dx, \ee where $\delta(x)$ is the
standard Dirac delta function. We call $\mu_{n, A_N}$ the normalized
spectral measure associated to $A_N$.

\begin{defi}[Normalized empirical spectral
distribution]\label{defi:nesd} Let $A_N \in \Omega_{n,N}$ have eigenvalues $\lambda_N \ge \cdots \ge
\lambda_1$. The normalized empirical spectral distribution (the
empirical distribution of normalized eigenvalues) $F_n^{A_N/\sqrt{N}}$
is defined by \be F_n^{A_N/\sqrt{N}}(x) \ = \frac{\#\{i \le N:
\lambda_i/\sqrt{N} \le x\}}{N}. \ee \end{defi}

As $F_n^{A_N/\sqrt{N}}(x) = \int_{-\infty}^x \mu_{n, A_N}(t)dt$, we see
that $F_n^{A_N/\sqrt{N}}$ is the cumulative distribution function
associated to the measure $\mu_{n, A_N}$. 

We are interested in the behavior of a typical $F_n^{A_N/\sqrt{N}}$ as we vary $A_N$ in our ensembles $\Omega_{n,N}$ as
$N\to\infty$. Our main results are that for each $n$, $F_n^{A_N/\sqrt{N}}$ converges
to the cumulative distribution function of a probability distribution.

As there is a
one-to-one correspondence between $N\times N$ real symmetric
palindromic Toeplitz matrices whose first row is $n$ copies of a palindrome and $\R^{N/2^n}$, we may study the more
convenient infinite sequences. Thus our outcome space is $\Omega_{n,\N}
= \{b_0,b_1,\dots\}$, and if $\omega = (\omega_0,\omega_1,\dots) \in
\Omega_{n,\N}$ then \be {\rm Prob}(\omega_i \in [\alpha_i,\beta_i])\ =\
\int_{\alpha_i}^{\beta_i} p(x_i)dx_i.\ee We denote elements of
$\Omega_{n,\N}$ by $A$ to emphasize the correspondence with matrices,
and we set $A_N$ to be the real symmetric palindromic Toeplitz
matrix obtained by truncating $A = (b_0,b_1,\dots)$ to
$(b_0,\dots,b_{N/2^n-1})$. We denote the probability space by
$(\Omega_{n,\N},\f_{n, \N},\p_{n, \N})$.

To each integer $m \ge 0$ we define the random variable $X_{m,n;N}$ on
$\Omega_{n,\N}$ by \be X_{m,n;N}(A) \ = \ \int_{-\infty}^\infty x^m
dF_n^{A_N/\sqrt{N}}(x); \ee note this is the $m$\textsuperscript{th}
moment of the measure $\mu_{n, A_N}$.

We investigate several types of convergence.

\ben

\item (Almost sure convergence) For each $m$, $X_{m,n;N} \to X_{m,n}$
almost surely if \be \p_{n, \N}\left(\{A \in \Omega_{n,\N}: X_{m,n;N}(A) \to
X_{m,n}(A)\ {\rm as}\ N\to\infty\}\right) \ =  \ 1; \ee

\item (In probability) For each $m$, $X_{m,n;N} \to X_{m,n}$ in
probability if for all $\gep>0$, \be
\lim_{N\to\infty}\p_{n, \N}(|X_{m,n;N}(A) - X_{m,n}(A)|
> \gep) \ = \ 0;\ee

\item (Weak convergence) For each $m$, $X_{m,n;N} \to X_{m,n}$ weakly if
\be \p_{n, \N}(X_{m,n;N}(A) \le x) \ \to \ \p(X_{m,n}(A) \le x) \ee as
$N\to\infty$ for all $x$ at which $F_{X_{m,n}}(x) = \p(X_{m,n}(A) \le x)$ is
continuous.

\een

Alternate notations are to say \emph{with probability 1} for almost
sure convergence and \emph{in distribution} for weak convergence;
both almost sure convergence and convergence in probability imply
weak convergence. For our purposes we take $X_{m,n}$ as the random
variable which is identically $M_{m,n}$, the  limit of the average $m$\textsuperscript{th} moments (i.e., $\lim_{N\to\infty} M_{m,n;N}$), which we show below uniquely determine a probability distribution.

Our main tool to understand the $F_n^{A_N/\sqrt{N}}$ is the Moment
Convergence Theorem (see \cite{Ta} for example); while the analysis in \cite{MMS} was simplified by the fact that the convergence was to the standard normal, similar arguments hold in our case as the growth rate of the moments of our limiting distribution implies that the moments uniquely determine a probability distribution.

\begin{thm}[Moment Convergence Theorem]\label{thm:momct} Let $\{F_N(x)\}$ be a
sequence of distribution functions such that the moments \be M_{m;N}
\ = \ \int_{-\infty}^\infty x^m dF_N(x) \ee exist for all $m$. Let $\{M_m\}_{m=1}^\infty$ be a sequence of moments that uniquely determine a probability distribution, and denote the cumulative distribution function by $\Psi$. If
$\lim_{N\to\infty} M_{m,N} = M_m$ then $\lim_{N\to\infty} F_N(x) =
\Psi(x)$. \end{thm}

\begin{defi}[Limiting spectral distribution]\label{defi:lsd}
If as $N\to\infty$ we have $F_n^{A_N/\sqrt{N}}$ converges in some
sense (for example, in probability or almost surely) to a distribution $F_n$,
then we say $F_n$ is the limiting spectral distribution of the
ensemble.
\end{defi}

\subsection{Results}\label{sec:results}

Our main result concerns the limiting behavior (as a function of the palindromicity $n$) of the $\mu_{n, A_N}$ for generic $A_N$ as $N\to\infty$. We analyze these limits using the method of moments. Specifically, for each $A_N$ we calculate the moments of $\mu_{n, A_N}$ by using the Eigenvalue Trace Lemma to relate the $m$\textsuperscript{th} moment to the trace of $A_N^m$. We show the average $m$\textsuperscript{th} moment tends to the $m$\textsuperscript{th} moment of a distribution with unbounded support. By analyzing the rate of convergence, we obtain results on convergence in probability and almost sure convergence, which we recall below.

Specifically, our main result is the following.

\begin{thm}\label{thm:mainweakstrong} {\rm \textbf{(Convergence in Probability and Strong Convergence)}} Let $n$ be a fixed positive integer, and for each $N$ a multiple of $2^n$ consider the ensemble of real symmetric $N\times N$ palindromic Toeplitz matrices whose first row is $2^n$ copies of a fixed palindrome (see \eqref{def:descriptionhighlytoepensemble} for an example), where the independent entries are independent, identically distributed random variables arising from a probability distribution $p$ with mean 0, variance 1 and finite higher moments. Then as $N\to\infty$ the measures $\mu_{n, A_N}$ (defined in \eqref{eq:normspacingmeasure}) converge weakly to a limiting spectral measure with unbounded support. If additionally $p$ is even, then the measures converge strongly to the limiting spectral measure. \end{thm}

As in other related ensembles, it is very difficult to obtain closed form expressions for the general moments of the limiting spectral measure. We can, however, analyze the moments well enough to determine the limiting distribution has unbounded support; in fact, as the following theorem shows, it has fatter tails than previously studied ensembles.

\begin{thm}\label{thm:fattails}{\rm \textbf{(Fat Tails)}} Consider the ensemble from Theorem \ref{thm:mainweakstrong}. For any fixed $n \ge 2$, the moments grow faster than the corresponding moments of the standard normal, and thus our density has fatter tails than the standard normal. Specifically, if $f_n$ denotes our density and $\phi$ denotes the density of the standard normal, for all $B$ sufficiently large we have $\int_B^\infty f_n(x) dx > \int_B^\infty \phi(x)dx$ (i.e., the probability of observing a value at least $B$ is larger for our limiting spectral distribution that the corresponding value for the standard normal). \end{thm}

Similar to other papers, our analysis is based on analyzing the contribution from various matchings. We are able to analyze in complete detail the case when we have all adjacent matchings. Based on numerical investigations and some theoretical calculations, we believe that all configurations contribute equally; see Conjecture \ref{conj:configurationconjequality} for a precise statement. If true, this would allow us to sharpen Theorem \ref{thm:fattails}.

\begin{thm}\label{thm:sharpenedresult} Assuming Conjecture \ref{conj:configurationconjequality}, if $M_{2m,n}$ denotes the $2m$\textsuperscript{th} moment of the limiting spectral measure of our ensemble for a given $n$, then \be M_{2m,n} \ \gg \  \frac{2^{m+n}}{m} \cdot (2m-1)!!.\ee  The limiting spectral measure thus has unbounded support, and fatter tails than the standard normal (or in fact any of the known limiting spectral measures arising from an ensemble where the independent entries are chosen from a density whose moment generating function converges in a neighborhood of the origin). \end{thm}

The rest of the paper is organized as follows. We first establish some basic results about our ensembles and the associated measures in \S\ref{sec:diophform}. We then analyze the even moments in detail in \S\ref{sec:speccharhptm}, and prove our convergence claims in \S\ref{sec:converge}. We give the proof on the vanishing Diophantine obstructions for highly palindromic Toeplitz matrices and discuss the configurations of the different matchings of highly palindromic Toeplitz matrices; our conjecture on the contribution from various matchings is discussed in detail in the appendices. While it is difficult to isolate the exact value of these moments, we are able to analyze these moments well enough to prove our convergence claims and to have some understanding of the limiting spectral measure. The situation is different for both the fourth moment for any palindromicity, and we determine the exact values in \S\ref{sec:fourthdouble}.

While many of the arguments in this paper about general properties of the moments and convergence are straightforward generalizations of those in \cite{HM, MMS}, the higher degree of palindromicity creates numerous technical difficulties which must be overcome if we want to compute actual values of the moments. In particular, the combinatorics becomes significantly harder, as can be seen by the length of \S\ref{sec:fourthmomentandfattertails}, which is the heart of this paper.

It is worth remarking on the role numerical investigations played in our analysis. These studies were essential in highlighting the key features and illuminating the structure and the combinatorics. Unfortunately the rate of convergence is significantly slower than in the ensembles in \cite{HM, MMS}, and it is thus non-trivial to extract useful data from these simulations. We end the paper with a brief discussion in Appendix \ref{sec:numerics} on a fast way to test conjectured answers to the difficult combinatorial problems.

The results in this paper arise by increasing the palindromicity of the matrices, which leads to highly patterned matrices. A natural question is what happens if instead we weaken the structure. In a sequel work, M. Kolo$\check{{\rm g}}$lu, G. Kopp and S. J. Miller \cite{KKM} do just that. They study the ensemble of real symmetric period $m$--circulant matrices, where each diagonal is now periodic with period $m$ (a real symmetric circulant matrix has constant diagonals that wrap around). This gives an interpolation from the highly structured symmetric circulant matrices, whose limiting eigenvalue density is a Gaussian, to the ensemble of all real symmetric matrices, whose limiting eigenvalue density is a semi-circle. The combinatorics can be reinterpreted as a counting problem in algebraic topology, and closed form expressions are obtained. The limiting spectral measure is the product of a Gaussian with a polynomial of degree $2m-2$, and rapidly converges to the semi-circle as $m$ tends to infinity.



\section{Diophantine Formulation}\label{sec:diophform}

In this section we begin our analysis of the moments. We prove some combinatorial results which restrict the number of configurations which can contribute a main term; we then analyze the potential main terms in the following section.

Recall that for each matrix $A_N \in \Omega_{n,N}$ we associate a probability measure by placing a point mass of size $1/N$ at each of its normalized eigenvalues $\lambda_i (A_N)$:
\be \mu_{n, A_N} (x) dx\ :=\ \frac{1}{N} \sum^{N}_{i=1} \delta \left(x - \frac{\lambda_i (A_N)}{\sqrt{N}}\right) d x, \ee
where $\delta (x)$ is the Dirac delta function. Thus the $k$\textsuperscript{th} moment of $\mu_{n, A_N} (x)$ is
\be\label{eq:formulaMkNAN} M_{k,n;N}(A_N)\ :=\ \int^{\infty}_{-\infty} x^k \mu_{n, A_N} (x) d x\ =\ \frac{1}{N^{k /2  +1}} \sum^{N}_{i=1} \lambda^k_i (A_N). \ee
The expected value of the $k$\textsuperscript{th} moment of the $N\times N$ matrices in our ensemble, found by averaging over the ensemble with each $A_N$ weighted by \eqref{eq:probspace} and using the Eigenvalue Trace Lemma, is
\bea\label{eq:coseqetl} M_{k,n;N}\ :=\  \E [ M_{k,n;N}(A_N)] & \ = \ & \frac{1}{N^{k/2 + 1}} \sum_{1\leq i_1, \dots, i_k \leq N} \E [a_{i_1 i_2} a_{i_2 i_3} \cdots a_{i_k i_1}]  \nonumber\\ & \ = \ & \frac{1}{N^{k/2 + 1}} \sum_{1\leq i_1, \dots, i_k \leq N} \E [b_{\psi (i_1, i_2)} b_{\psi (i_2, i_3)} \cdots b_{\psi (i_k, i_1)}],  \nonumber\\  \eea where from \eqref{eq:probspace} the expectation equals \be \E[b_{\psi (i_1, i_2)} b_{\psi (i_2, i_3)} \cdots b_{\psi (i_k, i_1)}] \ := \ \int \cdots \int b_{\psi (i_1, i_2)} b_{\psi (i_2, i_3)} \cdots b_{\psi (i_k, i_1)} \prod_{i=0}^{\frac{N}{2^n}-1} p(b_i)db_i. \ee We let $M_{k,n}$ be the limit of the average moments; thus  \be\label{eq:formulaMkNANlimit} M_{k,n}\ :=\ \lim_{N\to\infty} M_{k,n;N}; \ee we will prove later that these limits exist.

Our goal is to understand the $M_{k,n}$, i.e., the limiting behavior of the moments in these ensembles. We use Markov's Method of Moments, which we summarize below. This is a standard method for proving results in the subject; a nice explicit summary of this method begins Section 3 of \cite{BB1}.

\bi

\item We first show $M_{m,n} = \lim_{N\to\infty} M_{m,n;N} = \lim_{N\to\infty} \E[M_{m,n;N}(A_N)]$ exists for $m$ a positive integer, with the $M_{m,n}$'s satisfying Carleman's condition: $\sum_{m=1}^\infty M_{2m}^{-1/2m}$ $=$ $\infty$. As these are the moments of the empirical distribution measures, this implies that the $M_{m,n}$'s are the moments of a distribution.

\item Convergence in probability follows from analyzing the second moment, namely showing ${\rm Var}(M_{m,n;N}(A_N) - M_{m,n})$ tends to zero as $N\to\infty$.

\item Almost sure convergence follows from showing the fourth moment tends to zero and then applying the Borel-Cantelli lemma.

\ei

We do the convergence calculations in \S\ref{sec:converge}; in this and the next few sections we determine the limiting behavior of the ensemble averages.

The odd moments are readily determined, as counting the degrees of freedom show the average odd moments vanish in the limit as $N \to \infty$.

\begin{lem}\label{lem:2.6} All the average odd moments vanish in the limit; i.e. $\lim_{N\to\infty} M_{2m+1,n;N}$ $ =0$
\end{lem}

\begin{proof}
For the $(2m+1)$\ts{st} moment, we consider $\E[b_{\psi (i_1, i_2)} b_{\psi (i_2, i_3)} \cdots b_{\psi (i_{2m+1}, i_1)}]$; we may write this as $\E[b_{\ell_1}^{r_1} \cdots b_{\ell_j}^{r_j}]$ with $r_1+\cdots+r_j=2m+1$ and the $b_\ell$'s distinct. As $2m+1$ is odd, at least one $b_\ell$ is raised to an odd power. If any of these occur to just the first power, then the expectation is zero as the $b$'s are drawn from a mean zero distribution. Thus at least one of the $b_\ell$'s above occurs at least three times, and every $b_\ell$ occurs at least twice. The maximum number of distinct $b_\ell$'s occurs when everything is matched in pairs except for one triple matching. Thus there are at most $m$ different $b_\ell$'s in our tuple, and the number of tuples is bounded independent of $N$. We have two degrees of freedom from the first matching of the $b_\ell$'s and one degree of freedom for each other matching,\footnote{For example, say $b_{\psi(i_1,i_2)} = b_{\psi(i_v,i_{v+1})}$, with $i_1$ our first index. Both $i_1$ and $i_2$ are free variables and we have $N$ choices for each; however, $i_v$ is not (it will have occurred in a matching before this point), and   $i_{v+1}$ is determined by requiring the two $b_\ell$'s under consideration to be equal. The number of choices for $i_{v+1}$ depends on $n$ (the larger $n$ is, the more diagonals work); what matters is that the number of choices for $i_{v+1}$ is bounded independently of $N$. Whenever we have a new pair, we have a new choice for the value of the link function, and thus gain a degree of freedom.} for a total of at most $m+1$ degrees of freedom. Thus the number of indices $i_1, \dots, i_k \in \{1,\dots,N\}$ that can contribute to the moment in \eqref{eq:coseqetl} for a given matrix is $O_n(N^{m+1})$ (where the big-Oh constant may depend on $n$, as the larger $n$ is the more choices we have for diagonals). As we divide by $N^{m+3/2}$ in \eqref{eq:coseqetl}, the odd moments are $O_n(N^{-1/2})$, and thus vanish in the limit as $N\to\infty$.
\end{proof}

%


%

\begin{cor}\label{cor:cortolemma26} For fixed $n$, as $N\to\infty$ there is no contribution to the average $2m$\ts{th} moment from any tuple where the $b_\ell$'s are not matched in pairs. \end{cor}

\begin{proof}
The corollary follows from a similar analysis as in Lemma \ref{lem:2.6}.
\end{proof}

From the above corollary, we see that in order to study the eigenvalues of our matrices we need to know how many different ways the $k=2m$ entries (the $a_{i_j i_{j+1}}$'s) in our tuples can be matched into $k/2=m$ pairs. Letting $r!! = r(r-2)(r-4)\cdots$, where the product stops at $1$ if $r$ is odd and $2$ if $r$ is even, we see there are at most $(2m-1)!!$ ways to match in pairs.\footnote{There are $\ncr{2m}{2}$ ways to choose the first pair, $\ncr{2m-2}{2}$ ways to choose the second and so on; we must divide by $m!$ as it does not matter which pair we call the first. The claim follows by elementary algebra. Alternatively we can prove this by induction. Assume there are $(2m-3)!!$ ways to match $2m-2$ objects in pairs. If we have $2m$ objects, there are $2m-1$ choices of an element to pair with the first element in our list, and then by induction there are $(2m-3)!!$ ways of pairing the remaining $2m-2$ elements.} Note $(2m-1)!!$ is the $2m$\textsuperscript{th} moment of the standard normal, and has the combinatorial interpretation of being the number of ways of matching $2m$ objects in $m$ pairs where order does not matter.  For each legitimate matching we obtain a system of $m$ equations, one for each pair of entries, for which the number of solutions is the contribution of the matching to the $2m$\textsuperscript{th} moment.

In order to understand the even moments, we need to know more about the permissible matchings, and how many choices of the indices lead to valid configurations. In the original case of the ensemble of real symmetric Toeplitz matrices \cite{HM}, the only way any two entries $b_\ell$ could match was for them to lie on the same diagonal or on the reflection of that diagonal over the main diagonal. That is, they matched if and only if \be |i_{p} - i_{p+1}|\ =\ | i_l - i_{l+1}|. \ee

For highly palindromic Toeplitz matrices, more relations give matchings (as seen in the investigation of palindromic matrices in \cite{MMS}). An entry for which the absolute value of the difference between its indices is in a given congruence class modulo $2^n$ can match with another entry if and only if it is in the same congruence class or its negative. That is, two entries $a_{i_{p} i_{p+1}}$ and $a_{i_l i_{l+1}}$ can be matched in a pair if and only if their indices satisfy one of the following relations:

\begin{enumerate}

\item there is a $C_1 \in \{ (- \lfloor \frac{|i_l - i_{l+1}|}{2^n} \rfloor + k - 1) \frac{N}{2^n} - 1\mid k\in\{1,\ldots, 2^n\} \}$ such that
\be  \label{eq:match1} | i_p - i_{p+1}|\ =\ | i_l - i_{l+1}| + C_1; \ee
\item  there is a $C_2 \in \{ (\lfloor \frac{|i_l - i_{l+1}|}{2^n} \rfloor + k) \frac{N}{2^n} \mid k\in\{1,\ldots, 2^n\}  \}$ such that
\be\label{eq:match2}  |i_p - i_{p+1}|\ =\ - |i_l - i_{l+1}|  + C_2; \ee as is standard, $\lfloor x \rfloor$ represents the largest integer at most $x$.

\end{enumerate}

As a consequence of \eqref{eq:match1} and \eqref{eq:match2}, for the matchings above there is some $C$ such that \be\label{eq:matchingeqsigns} i_p - i_{p+1} \ = \ \pm (i_l - i_{l+1}) + C. \ee As there are two choices for sign for each of the $m$ matchings, there are potentially $2^m$ cases that can contribute. We now prune down the number of possibilities greatly by showing only one case contributes in the limit, namely the case when all the signs are negative.

In the Toeplitz ensembles studied in \cite{HM} and \cite{MMS}, it was shown that any matching with a positive sign (i.e., as in \eqref{eq:matchingeqsigns}) in any pair contributes a lower order term to the moments, and thus it sufficed to consider the case where only negative signs occurred. A similar result holds here, which greatly limits the number of cases we need to investigate. Note by Lemma \ref{lem:2.6} we need only investigate the even moments.

\begin{lem}\label{thm:neg} Consider the contribution to the $2m$\textsuperscript{{\rm th}} moment from all tuples $(i_1, \dots$, $i_{2m})$ in which the corresponding $b_\ell$'s are matched in pairs. If an $a_{i_q i _{q+1}}$ is matched with an $a_{i_l i _{l+1}}$ with a positive sign (which means
\[  i_q - i_{q+1}\ =\  + (i_l - i_{l+1})  +C\]
for some $C$ as defined in \eqref{eq:match1} or \eqref{eq:match2}), then this matching contributes $O_m(1/N)$ to $M_{2m,n;N}$ and therefore the contribution of all but one of the $2^m$ choices for the $m$ signs vanishes in the limit as $N\to\infty$, with only the choice of all negative signs being able to contribute in the limit.
\end{lem}

\begin{proof} The argument is essentially the same as in \cite{MMS}. Briefly, the idea is that if there is ever a positive sign then we lose a degree of freedom, leading to a lower order contribution.

For any tuple $(i_1, \dots, i_{2m})$ in which the corresponding $b_\ell$'s are matched in pairs, there exist $k$ equations, one for each pairing, of the form
\be  \label{eq:match 3} i_q - i_{q+1}\ =\  \epsilon_l(i_l - i_{l+1})  + C_l \text{ where } \epsilon_l \ = \ 1\text{ or } -1. \ee
Let $x_1, x_2, \dots, x_{2m}$ denote the absolute value of the difference between two indices of each entry (so for $a_{i_l,i_{l+1}}$ it would be $x_l = |{i_1}-i_{l+1}|$), and let $ \tilde{x}_1 = i_1 - i_2,  \tilde{x}_2 = i_2 - i_3, \dots $ and $ \tilde{x}_{2m} = i_{2m} - i_1$ (i.e., the unsigned differences). It follows immediately that
\be\label{eq:tildesum} \sum^{2m}_{i=1} \tilde{x}_i\ =\ 0.\ee
Each $\tilde{x}_r$ can be expressed in two ways.  By breaking the absolute value sign in \eqref{eq:match1} or \eqref{eq:match2}, we have $ \tilde{x}_r = \eta_j x_j$ for some $j$ with $\eta_j = 1$ or $-1$. We can also express it through an equation like the one in \eqref{eq:tildesum} such that $ \tilde {x}_r = \epsilon_r  \tilde{x}_l + C_r$ for some $l$. Thus
\be\label{eq:tildex} \tilde{x}_r\ =\  \eta_j x_j = \epsilon_r \tilde{x}_l +C_l. \ee
Then since $\epsilon_r^2=1$,
\be\label{eq:oldeq220} \tilde{x}_l \ = \    \epsilon_r \eta_j x_j - \epsilon_r C_l.\ee
Note each absolute value of a difference occurs twice, as everything is matched in pairs. We therefore have \be  \sum^{2m}_{i=1} \tilde{x}_i \ = \ \sum^{m}_{j=1} [ \eta_j x_j +  (\epsilon_r \eta_j x_j - \epsilon_r C_n)] \  =  \  \sum^{m}_{j=1} (n_j x_j (1 + \epsilon_r) - \epsilon_r C_j)\ =\   0.\ee

If any $\epsilon_r = 1$, then the $x_j$'s are not linearly independent and we would have less than $m+1$ degree of freedom.\footnote{As in the proof of Lemma \ref{thm:neg}, the first pair gives us two degrees of freedom and each subsequent pair gives at most one degree of freedom. If the $x_j$'s are not linearly independent, there can be at most $m-1$ independent $x_j$'s, and thus at most $m$ degrees of freedom. } The contribution from such tuples to the moment in \eqref{eq:coseqetl} for a given matrix is therefore $O(1/N)$ (as we divide by $N^{m+1}$), which vanishes in the limit as $N\to \infty$ and can thus be safely ignored.\end{proof}


Lemma \ref{thm:neg} immediately implies

\begin{lem}\label{lem:old221222}
If the indices of $a_{i_l i_{l+1}}$ and $a_{i_q i_{q+1}}$ satisfy \eqref{eq:match1} for some $C_1$, then  $|i_q - i_{q+1}| = |i_l - i_{l+1}| + C_1$ implies
\be\label{eq:2.15}
\begin{cases}
i_q - i_{q+1}\ =\ - (i_l - i_{l+1}) + C_1  \\
i_q\ >\ \max \{i_{q+1}, i_{q+1} + C_1\} \\
\end{cases} \\
\text{ or }
\begin{cases}
i_q - i_{q+1}\ =\ - (i_l - i_{l+1}) - C_1  \\
i_q\ <\ \min \{i_{q+1}, i_{q+1} - C_1\}. \\
\end{cases}
\ee

Similarly, if the indices satisfy \eqref{eq:match2} for some $C_2$, then $|i_q - i_{q+1}| = -|i_l - i_{l+1}| + C_2$ implies
\be\label{eq:2.16}
\begin{cases}
i_q - i_{q+1}\ =\ - (i_l - i_{l+1}) + C_2\\
i_{q+1}\ <\ i_q\ <\ i_{q+1} + C_2, \\
\end{cases} \\
\text{ or }
\begin{cases}
i_q - i_{q+1}\ =\ - (i_l - i_{l+1}) - C_2\\
i_{q+1} - C_2\ <\ i_q\ <\ i_{q+1}.\\
\end{cases}
\ee
\end{lem}


Instead of considering each value of $C$ (either $C_1$ or $C_2$) individually, we will consider a pair of constants $C_1, C_2$ such that $C_1 + C_2 = N -1$. We claim that this removes some of the Diophantine obstructions that arise when evaluating \eqref{eq:2.15} or \eqref{eq:2.16} individually. Given an entry $a_{i_l i_{l+1}}$, we can associate each value of $C$ with one diagonal whose entries, generally denoted by $a_{i_r i_{r+1}}$, all equal $a_{i_l i_{l+1}}$. Except for the main diagonal, every other diagonal has fewer than $N$ entries and therefore the index $i_r \in \{a,\dots,b\}$ where $1\leq a<b \leq N$ rather than $i_r \in \{1, \dots, N\}$. Here we only need to restrict one of the two indices of $a_{i_r i_{r+1}}$ and the other one will automatically be determined. However, by considering $a_{i_r i_{r+1}}$ on a pair of diagonals associated with $C_1, C_2$, we can take the index $i_r$ (or $i_{r+1}$) to be any value between $1$ and $N$. Furthermore, except for $O(1)$ values, the first index of entries from the pair of diagonals associated with $C_1, C_2$ are distinct, and similarly for the second index. Therefore, if $a_{i_r i_{r+1}}$ is on the diagonal associated with $C_1$ and $a_{i^{'}_r i^{'}_{r+1}}$ is on the diagonal associated with $C_2$, then for some $a,b \in \{1,\dots,N\}$, we have:
\be
\begin{cases}
i_r \in \{a,\dots,b\} \\
i^{'}_r \in \{0,\dots,a\} \cup \{b,\dots,N\} \\
\end{cases}
\ee














\section{Properties of the Even Moments}\label{sec:speccharhptm}

In Lemma \ref{lem:2.6} we showed that the average odd moments vanish in the limit. In this section we analyze the even moments. While the low moments may be computed by brute force (we provide the computation for the fourth moment in \S\ref{sec:fourthmomentandfattertails} and discuss its consequences), similar to other ensembles we are unable to obtain nice closed form expressions for the higher moments in general, although a combination of numerical simulations and some partial results suggest the answer for the doubly palindromic case; see Appendices \ref{sec:numerics} and \ref{sec:configurationconj}.


\subsection{General Properties}

We first handle the zeroth and second moments, and then turn to the higher moments.

\begin{lem} Assume that $p$ has mean 0, variance 1 and finite higher moments, and fix the degree of palindromicity $n$. Notation as above (see \eqref{eq:formulaMkNAN} and \eqref{eq:formulaMkNANlimit}), for all $A_N$ we have $M_{0,n;N}(A_N) = 1$ and $M_{2,n;N}(A_N) = 1$, which implies the average moments in the limit are both 1 (explicitly, $M_{0,n} = 1$ and $M_{2,n} = 1$). \end{lem}

\begin{proof}
From \eqref{eq:formulaMkNAN}, we see $M_{0,n,N}(A_N) = 1$. For the second moment, we have \bea M_{2,n;N} & \ = \ & \frac{1}{N^2} \sum_{1 \le i_1,i_2 \le N} \E(a_{i_1 i_2}\cdot a_{i_2 i_1}) \nonumber\\ & \ = \ & \frac{1}{N^2} \sum_{1 \le i_1,i_2 \le N} \E(a_{i_1 i_2}^2) \ = \ \frac{1}{N^2} \sum_{1 \le i_1,i_2 \le N} \E(b_{\psi(i_1,i_2)}^2). \eea Since we choose the $b$'s from a distribution with mean zero and variance 1, the expected value above is just the variance (which is $1$), and hence $M_{2,n;N} = 1$, which implies $M_{2,n} = \lim_{N\to\infty} M_{2,n;N} = 1$. \end{proof}

We now consider the general even moments. By Corollary \ref{cor:cortolemma26}, the only contributions to the moments $M_{2m,n;N}$ (see \eqref{eq:coseqetl}) that survive as $N\to\infty$ is when the $a_{i_j i_{j+1}}$'s are matched in pairs. There are $(2m-1)!!$ such matchings; we need to determine the contribution of each matching to $M_{2m,n;N}$.

Each of the $(2m-1)!!$ matchings, hereafter referred to as a \textbf{configuration}, leads to a system of $m$ equations of the form \eqref{eq:2.15} or \eqref{eq:2.16} (with the $C$'s coming from \eqref{eq:match1} and \eqref{eq:match2}), for which each distinct solution gives us one possible choice for the tuples $(i_1, \dots, i_{2m})$ and contributes one to the sum. The analysis is completed by counting how many valid configurations there are (or at least determining the main term).

Determining the exact value is complicated by the fact that there are many ways for an $a_{i_j i_{j+1}}$ and an $a_{i_v i_{v+1}}$ to be paired; they must correspond to the same $b_\ell$, but there are many diagonals each can lie on (with the number of diagonals growing with $n$). Fortunately, we can obtain a weak bound depending on $n$ that nevertheless suffices to prove the existence of a limiting spectral measure. By standard arguments, it suffices to show the average even moments converge as $N\to\infty$ to a sequence satisfying Carleman's condition, and then perform a similar analysis on the variance (for convergence in probability) or the fourth moment (for almost sure convergence). We leave the convergence issues to \S\ref{sec:converge}, and instead prove the existence of the limits.

\begin{lem}\label{lem:existenceevenmoments} For fixed $n$, $M_{2m,n}$ exists and \be M_{2m,n}\ =\ \lim_{N\to\infty} M_{2m,n;N} \le (2\cdot 2^n)^m (2m-1)!!, \ee which implies the $M_{2m,n}$ satisfy Carleman's condition.
\end{lem}

\begin{proof} Fix $n$ and $m$. Consider one of the $(2m-1)!!$ pairings where the $\{a_{i_l i_{l+1}}\}_{l=1}^{2m}$ are matched in pairs. We have $m$ equations. For our system of equations we must choose $m$ values for the $C$'s, and in each equation by \eqref{eq:match1} and \eqref{eq:match2} there are at most $2 \cdot 2^n$ choices for a $C$. By our previous analysis, at most $m+1$ of the indices $i_1, \dots, i_{2m}$ are free. We must analyze the contribution from each choice of the indices in \eqref{eq:coseqetl}. Assume first each $a_{i_l i_{l+1}}$ is matched with a unique $a_{i_p i_{p+1}}$. In this case, the contribution of this choice of indices to \eqref{eq:coseqetl} is just the product of $m$ copies of the expected value of the second moment of the probability density $p$; as the second moment is 1, each of these adds 1 to \eqref{eq:coseqetl}, and there are clearly at most $N^{m+1}$ choices of indices.

If all of the $a_{i_l  i_{l+1}}$ are not uniquely matched with another $a_{i_p i_{p+1}}$, then it is possible to have a larger contribution than $1$ to \eqref{eq:coseqetl}, as the product of the expected values could involve fourth, sixth, eighth, ..., $2m$\textsuperscript{th} moments. Let $p_m$ be the maximum of the absolute values of the first $2m$ moments of $p$. The contribution in this case is at most $p_m^m$; while this is growing with $m$, the number of indices that can contribute this is at most $O(N^m)$ by our earlier arguments, as we showed we lose at least one degree of freedom when items are matched in  more than pairs.

As we divide by $N^{m+1}$ in \eqref{eq:coseqetl}, we find \be M_{2m,n;N} \ \le \ \frac{(2m-1)!! \cdot (2 \cdot 2^n)^m}{N^{m+1}} \left[ N^{m+1} \cdot 1 + O(N^m) \cdot p_m^m\right]. \ee Taking the limit as $N\to\infty$ yields \be M_{2m,n} \ = \ \lim_{N\to\infty} M_{2m,n;N} \ \le \  (2\cdot 2^n)^m \cdot (2m-1)!!. \ee

The existence of the limit is proved analogously to \cite{BB1,BDJ,HM,MMS} (see for example Theorem 2.6 of \cite{HM}); now that we know $M_{2m,n}$ is bounded, it is easy to see that the main term of the contribution from each possible configuration is independent of $N$.

It remains to show that the $M_{2m,n}$ satisfy Carleman's condition by showing the sum of the reciprocals of their $2m$\ts{th} roots diverge. Trivial estimation suffices. As $(2m-1)!! <  (2m)^{2m}$, we have $(2m-1)!!^{-1/2m} > 1/2m$, and thus \be \sum_{m} M_{2m,n}^{-1/2m} \ > \ \sum_m (2 \cdot 2^n)^{-1/2} \cdot \frac1{2m}. \ee The latter sum is a multiple of the harmonic sum and diverges, completing the proof.
\end{proof}













\section{Convergence}\label{sec:converge}

In \S\ref{sec:diophform} and \S\ref{sec:speccharhptm} we showed that the limit of the average moments exist as $N\to\infty$, culminating in Lemma \ref{lem:existenceevenmoments} where we proved that the moments grow slowly enough to uniquely determine a probability distribution. We now show convergence in probability, and if $p(x)$ is even we prove almost sure convergence. As these arguments closely follow those in \cite{HM,MMS}, we concentrate on the novelties introduced by the higher palindromicity. We conclude by obtaining lower bounds for the moments. These bounds imply that our limiting distributions have unbounded support and fatter tails than the standard normal (possibly the fattest tails observed from a random matrix ensemble arising from entries chosen independently from a distribution whose moment generating function converges in a neighborhood of the origin).

\subsection{Convergence in probability}

We will prove our probability measures converge in probability. We follow the arguments in \cite{HM, MMS}. We begin by defining our random variables. Let $A$ be a sequence of real numbers to which we associate an $N\times N$ real symmetric Toeplitz matrix with $2^n$ palindromes, which we denote by $A_N$. Thus we may view $A$ as $(b_0, b_1, b_2, \dots)$, and we form $A_N$ by considering the initial segment of length $N/2^{n+1}$, taking that as the first half of our palindrome, and then building the matrix by having $2^n$ palindromes in the first row.

Let $X_{m,n;N}(A)$ be a random variable that equals the $m\textsuperscript{th}$ moment of $A_N$ (so $X_{m,n;N}(A) = M_{m,n;N}(A_N)$), and set $M_{m,n;N}$ to the $m\textsuperscript{th}$ moment averaged over the ensemble as above (so $M_{m,n;N} = \E[X_{m,n;N}]$).

Thus, we have convergence in probability if for all $\epsilon>0$
\be
\lim_{N\to\infty}\mathbb{P}_{\N}(\{A\in\Omega_{\N}:|X_{m,n;N}-X_{m,n}|>\epsilon\})\ = \ 0.\ee
Using Chebyshev's inequality and the fact that ${\rm Var}(Y) = \E[(Y-\E[Y])^2] = \E[Y^2]-\E[Y]^2$, we have
\bea & &
\mathbb{P}_{\N}(\{A\in\Omega_{\N}:|X_{m,n;N}-\E[X_{m,n;N}]|>\epsilon\}) \nonumber\\ & & \ \ \ \ \ \ \ \ \ \ \ \ \ \ \ \ \le \ \frac{\E[(X_{m,n;N}-M_{m,n;N})^2]}{\epsilon^2}.\nonumber \\
& & \ \ \ \ \ \ \ \ \ \ \ \ \ \ \ \ = \ \frac{\E[X_{m,n;N}^2]-M_{m,n;N}^2}{\epsilon^2}. \eea
Thus, it suffices to show
\be\lim_{N\to\infty} (\E[X_{m,n;N}^2]-M_{m,n;N}^2)\ =\ 0\ee to prove convergence in probability.

We have
\bea \E[X_{m,n;N}^2] & \ = \ &
\frac{1}{N^{m+2}} \sum_{1 \le i_1,\dots,i_m \le N} \nonumber\\ & & \
\ \ \ \ \ \times \  \sum_{1 \le j_1,\dots,j_m \le N} \E[b_{|i_1-i_2|} \cdots b_{|i_m-i_1|} b_{|j_1-j_2|} \cdots b_{|j_m-j_1|} ] \nonumber,\\
M_{m,n;N}^2 & = & \frac{1}{N^{m+2}} \sum_{1 \le i_1,\dots,i_m \le N} \E[ b_{|i_1-i_2|} \cdots b_{|i_m-i_1|}] \nonumber\\
& & \ \ \ \ \ \ \times \ \sum_{1 \le j_1,\dots,j_m \le N}\mathbb{E}[b_{|j_1-j_2|} \cdots b_{|j_m-j_1|} ]. \eea

We can break this up into two cases.  If the entries of the $i$ diagonals are entirely distinct from those of the $j$ diagonals, then the contribution to $\E[X_{m,n;N}^2]$ and to $M_{m,n;N}^2$ will clearly be the same.  Thus, we need to approximate the contribution from the cases where there are one or more shared diagonals.  The degree of freedom arguments of \cite{HM} immediately apply here, though our big-Oh constants will now depend on the value of $2^n$ as we now have many more $C$-vectors to which we apply these arguments.  Thus, as $N\to\infty$ these two quantities will converge, and convergence in probability (and thus also weak convergence) follow.

\subsection{Almost Sure Convergence}

We assume that $p(x)$ is even for convenience, and use the same notation as above; in particular, \be M_{m,n}\ =\ \lim_{N\to\infty} M_{m,n;N} \ = \ \lim_{N\to\infty} \E[X_{m,n;N}(A)].\ee

Almost sure convergence follows from showing that as $N\to\infty$ the event
\[\{A\in\Omega_{\N}\ :\ \lim_{N\to\infty} X_{m,n;N}(A)\to M_{m,n}\}\] occurs with probability one for all non-negative integers $m$.

By the triangle inequality, we have that
\be |X_{m,n;N}(A) - M_{m,n}|\  \leq\ |X_{m,n;N}(A) - M_{m,n;N}| + |M_{m,n;N}-M_{m,n}|.\ee We have already shown that $\lim_{N\to\infty}|M_{m,n;N}-M_{m,n}|=0$, so we need only show that $|X_{m,n;N}(A) - M_{m,n;N}|$ almost surely tends to zero. Clearly  $\E[X_{m,n;N}(A) - M_{m,n;N}]=0$, and we can modify the arguments in \cite{HM} to show that the fourth moment of $X_{m,n;N}(A) - M_{m,n;N}$ is $O_{m,2^n}(\frac{1}{N^2})$. All of the degree of freedom arguments can be applied directly for each $C$-vector.

However, Theorems 6.15 and 6.16 of \cite{HM} require greater care as these use more than degree of freedom arguments. Fortunately, equations (50) and (51) in \cite{HM} hold for any of our $C$-vectors, so a similar result holds in this case. We then apply Chebyshev's inequality to find
\be \p_{n, \N}(|X_{m,n;N}(A) - M_{m,n;N}| \ge \gep) \ \le \ \frac{\E[|X_{m,n;N}(A)
- M_{m,n;N}|^4]}{\gep^4} \ \le \ \frac{C_{m,2^n}}{N^2 \gep^4}. \ee

Finally, applying the Borel-Cantelli Lemma shows that we have convergence everywhere except for a set of zero probability, thus proving almost sure convergence.






\section{Adjacent Matchings, Even Moments and the Tail}\label{sec:fourthmomentandfattertails}

As any distribution with finite mean and variance can be normalized to have mean 0 and variance 1, if the distribution is even then the fourth moment is the first moment to show the `shape' of the distribution, and thus merits special consideration. We analyze the fourth moment in detail below. We first prove  that the adjacent and non-adjacent matching configurations contribute equally. We then compute the contribution from the adjacent case in \S\ref{sec:fourthdouble}. We are able to compute the contribution from the adjacent case in the doubly palindromic case for any moment; if all configurations contributed equally (which we believe is the case) then the $2m$\textsuperscript{th} moment would just be $(2m-1)!!$ times the contribution from the adjacent configuration.

\subsection{Fourth Moment Configurations}

\begin{lem}[Equal Contribution - Fourth Moment]\label{thm:4momVP} The non-adjacent configuration and the adjacent configuration contribute equally to the fourth moment. \end{lem}

\begin{figure}[htb]
\begin{center}
\scalebox{.6}{\includegraphics{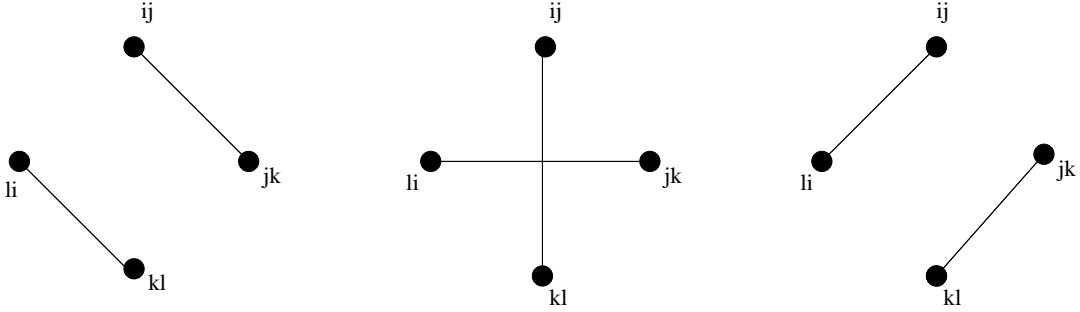}}
\caption{\label{fig:1}The adjacent and the non-adjacent configurations of the fourth moment. By relabeling we see the first and third are equivalent.}
\end{center}\end{figure}

\begin{proof} For brevity and notational simplicity, we will only present the calculation for the non-adjacent case. The adjacent case can be analyzed analogously and yields the same main term. That is, we consider the configuration with the following matchings:
\be
\begin{cases}
a_{ij} \ =\ a_{kl} \\
a_{jk}\ =\ a_{li}.\\
\end{cases}
\ee
From the above system of equations relating the matchings, we obtain the corresponding system of equations for the indices:
\be
\begin{cases}
|i - j|\ =\ \pm |k - l|  + A \\
|j - k|\ =\ \pm |l - i|  + B. \\
\end{cases}
\ee
Applying Lemma \ref{thm:neg}, we need only consider the case where
\be
\begin{cases}
i - j\ = \ - (k - l)  + A^{'} \\
j - k\ =\  - (l - i)  + B^{'} \\
\end{cases}
\ee where either $A^{'} = A$ or $A^{'} = - A$, and similarly for $B^{'}$, depending on how the absolute value equations resolve. Moreover, we see that
\be A^{'} + B^{'}\ =\ i - j + k - l + j - k + l - i\ =\ 0. \ee
This implies that $A$ and $B$ must be of the same form, either as in \eqref{eq:match1} or \eqref{eq:match2}. If $A$ is of the form $C_2$ in \eqref{eq:match2}, then it follows immediately that $A = B$, whereas if $A$ is of the form $C_1$ in \eqref{eq:match1}, then it can either be that $A = B$ or $A = -B$. For each $A$ we have a system of two equations with four unknowns so we can always pick at least two free indices. For convenience, we specify $i$ and $j$ as these free indices by choosing $a_{i j}$. Moreover, we assume that we only pick $a_{i j}$ in the lower diagonal half of the matrix so that $i > j$. By the symmetry of the matrix, picking $a_{i j}$ in the upper diagonal half would follow contribute equally.

We first consider the case where $A$ is of the form $C_2$ in \eqref{eq:match2}, and thus $A = B = C_2$ for some $C_2$. We then find
\begin{equation}
\begin{cases}
| k - l | \ = \ - | i - j | + C_2 \\
|j - k| \ = \ - |l - i|  + C_2 \\
\end{cases}
\Longrightarrow
\begin{cases}
k - l \ = \ - (i - j) + C_2 \\
k > l \\
j - k \ = \ - (l - i) - C_2\\
i - C_2 < l < i.\\
\end{cases}
\end{equation}

We now consider $A$ of the form $C_1$ in \eqref{eq:match1} where $C_1 + C_2= N -1$. The value $C_1$ is unique for each choice of $C_2$ and the contribution from the pair $(C_1, C_2)$ complements nicely one another as we show below. We see that $A = \pm B = C_1$. We then have
\begin{equation}
\begin{cases}
| k - l | \ = \ | i - j | + C_1 \\
|j - k| \ = \   |l - i|  \pm C_1\\
\end{cases}
\Longrightarrow
\begin{cases}
k - l \ = \ - (i - j) - C_1\\
k < l \\
j - k \ = \ - (l - i) + C_1\\
l < i \textbf{ or } l >  i + C_1. \\
\end{cases}
\end{equation}

Since we have already picked the first entry $a_{i j}$, we are left to choose the entry $a_{k l}$. Our choice of $C_1$ (or complementary $C_2$ since the pair is unique) indicates the diagonals that $a_{li}$ lies on, which gives us the restrictions on $\ell$. Finally, as only one of the indices $k$ or $l$ need to be specified (since the other is restricted by the diagonal), without loss of generality we choose $l$. We now use our previous analysis from \eqref{eq:oldeq220} and Lemma \ref{lem:old221222} to analyze the diagonals associated to $A=C_1$ and $A=C_2$. Since except for the main diagonal, every other diagonal has less than $N$ entries, if we choose $a_{k l}$ on the diagonal $A=C_2$, then there exist integers $a, b \in \{1, \dots, N \}$ such that $a \leq l \leq b$ and either $a$ or $b$ must be $N$. If we choose $a_{k l}$ on the diagonal $A=C_1$, then $l \leq a$ or $b \leq l$.

\begin{enumerate}

\item On the diagonal associated with $A = C_2$:
\be l \in \{i-C_2, \dots,  i\} \cap \{a, \dots, b\}. \ee

\item On the diagonal associated with $A = C_1$:
\be l \in (\{0,\dots,i\} \cup \{i+C_1, \dots, N\}) \cap (\{0, \dots, a\} \cup \{b, \dots, N\}). \ee

\end{enumerate}

Therefore, there are exactly $C_2$ out of $N+1$ values of $i$ we can pick (or exactly $C_1$ out of $N+1$ value of $i$ we cannot pick).  Since we have $N^2$ choices for picking the initial entry $a_{i j}$, the contribution to the fourth moment from the pair $(C_1, C_2)$ is given by
\be \label{thm: form1} N^2 \cdot C_2 \ = \ \left(\frac{N^3}{2^n}\right) \left(- \bigg\lfloor \frac{|i_1 - i_{2}|}{2^n} \bigg\rfloor + k \right). \ee
This contribution only depends on the initial choice of $a_{i j}$ and the choice of $A$.

Repeating this analysis for the adjacent case and summing over all possible choices of $A$ of the form $C_1$, we obtain the same contribution to the fourth moment from either configuration.
\end{proof}

\subsection{The Adjacent Case}\label{sec:adjacent}

In this section we analyze the Adjacent Case in detail, as this configuration is easier to study and more easily generalized. Since the contributions to the fourth moment from the adjacent and non-adjacent configurations are identical, this allows us to calculate the fourth moment for any number of palindromes. We can also calculate the contributions in the general Adjacent Case for \emph{any} even moments for the doubly palindromic Toeplitz matrix. In principle we could use the same ideas to calculate by brute force the adjacent configurations of any even moment of an ensemble with a greater number of palindromes, but we could not find a closed form expression for these moments. We conjecture that all configurations contribute the same main term in the limit, and provide numerical support in Appendix \ref{sec:numerics}.

\subsubsection{Determining the Fourth Moment}\label{sec:fourthdouble}

For the fourth moment, we have four indices $i,j,k,$ and $l$, and we consider an adjacent matching where
\[a_{ij} = a_{jk},\hspace{10 mm}a_{kl} = a_{li}.\]

We think about this as follows. A pair $i$ and $j$ gives us a matrix element $a_{ij}$; we want to find all pairs $j$ and $k$ such that $a_{ij} = a_{jk}$. This could happen by having the two on the same diagonal, or it could happen that $a_{jk}$ is on a palindromically equivalent diagonal. As the formula for the fourth moment of our matrix $A_N$ involves division by $N^3$, we need only worry about situations where we have on the order of $N^3$ tuples. Clearly we may choose $i$ and $j$ freely. The matching then forces there to be on the order of 1 choice for $k$ (the exact answer depends on $n$, the degree of palindromicity; what matters is that the answer is independent of $N$ in the limit), and on the order of $N$ choices for $l$. The last is important, as unless the number of choices of $l$ is proportional to $N$, we will obtain a negligible contribution from the matching $a_{ij}=a_{jk}$ and $a_{kl}=a_{li}$. Exploiting the symmetry of the matrix, this reduces to choosing $k$ so that $a_{ij}=a_{kj}$ and $a_{kl}=a_{il}$. That is, in addition to matching $a_{ij}$ and $a_{kj}$we want row $i$ and row $k$ to match well.

\begin{figure}[htb]
\begin{center}
\scalebox{.7}{\includegraphics{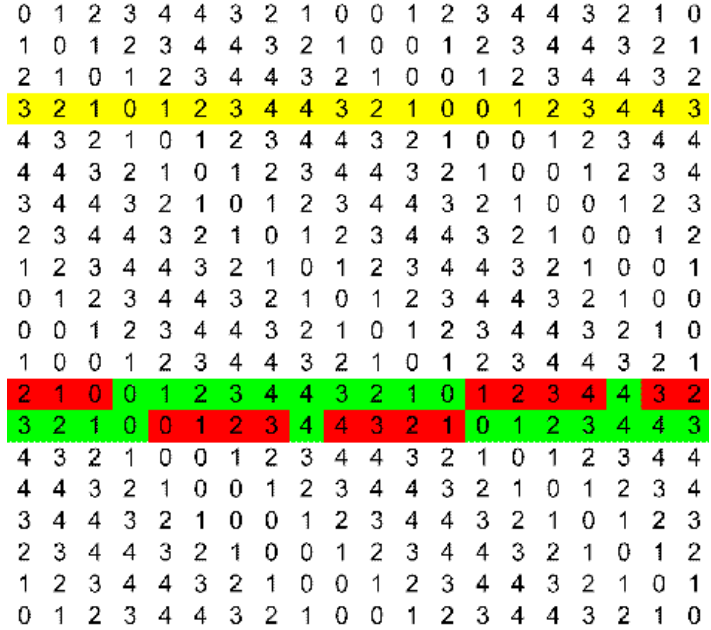}}
\caption{\label{fig:adjrows} An example highlighting matchings for $l$ in medium shading, with mismatching in dark shading. Note that any anomalous matchings won't contribute in the limit.}
\end{center}\end{figure}

We isolate some of the most useful features of our matrices in the following lemma. The proof follows immediately from the previous discussions and the structure of the matrices in our ensemble.

\begin{lem}\label{thm:properties} Fix $n$ and consider the ensemble of $N\times N$ real symmetric palindromic Toeplitz matrices with $2^n$ palindromes in the first row. The main diagonal is the only place (excluding the border of the matrix) where $b_0$ occurs once rather than twice. This implies the following useful properties.

\bi

\item Moving to the corresponding point in the next palindrome can require either moving $N 2^{-n} - 1$ elements when crossing the main diagonal or $N 2^{-n}$ elements otherwise.

\item Let $c\in\{1,2,\ldots,2^{-n}-1\}$. As pictured in Figure \ref{fig:adjrows}, a given row and the row $cN2^{-n}$ rows down from that given row do not match perfectly, but rather become unaligned when one row has reached the main diagonal but the other row has not. Moreover, the row $cN2^{-n} -1$ rows down starts out unaligned, but then becomes aligned in this same region. Furthermore, only rows of this form match up well with the original row.

\ei

\end{lem}

\begin{proof} The first item follows directly from the observation that the main diagonal is the only place where $b_0$ appears once rather than consecutively. We also see that, neglecting the first row which starts on the main diagonal, that the first elements of a row and one $cN2^{-n}$ rows away match initially. Moreover, they evolve the same way when moving from left to right, except when the first one hits the main diagonal, in which case it skips forward one place in the palindrome, in which case they do not match except possibly for repeating elements at the beginning/end or middle of palindromes, like $b_0$. However, once the second row hits the main diagonal, it also skips forward, and they become realigned. The case for rows $cN2^{-n} - 1$ rows away from each other is argued similarly.

To prove that no other rows match sufficiently well we need to show that there are only $O_n(1)$ matchings in any of the other rows. Suppose we do have a matching in one of the other rows. Since we can't be at the corresponding point in the palindrome, we must be at the other end of the palindrome. Unless these are the special repeating elements at the beginning or middle of a palindrome these two rows will evolve differently, so although there may be additional anomalous matchups, there will certainly not be more than four per palindrome, giving us the desired maximum of $O_n(1)$ possible matchings. If they are the special repeating elements, then the rows can match up well, but in this case there are at most five such elements per row, so we again have a lower order term.
\end{proof}

With this lemma in hand, we can now calculate the contribution from a specific constant, which will then allow us to calculate the contribution from the adjacent configuration.

\begin{lem} Let $c\in\{0,1,\ldots,2^{n-1}\}$ and $k = i + cN2^{-n}$. There are then
\be\left(\frac{2^n-c}{2^n}\right)^3 N^3+O_n(N^2)\ee good matchings, whereas if $k = i+ cN2^{-n}-1$, then there are
\be\frac{2^n c^2 - c^3}{2^{3n}} N^3 + O_n(N^2)\ee good matchings, where the big-Oh constants depend on $n$ (which is fixed).\footnote{The constants may be taken to depend on $c$ as well; however, as $n$ is fixed and $c \in \{0, \dots, 2^{n-1}\}$, we may take the maximum of all the constants and may replace $c$ dependence with $n$ dependence.}
\end{lem}

\begin{proof} We begin by noting that by Lemma \ref{thm:properties} above, choosing $k$ so that $a_{ij}$ and $a_{kj}$ are at corresponding points in a palindrome guarantees that $a_{ij} = a_{kj}$ and that there are $O_n(N)$ choices of $l$ satisfying $a_{kl} = a_{il}$, as desired. Moreover, if $a_{ij}$ and $a_{kj}$ aren't at corresponding points in the palindrome, then there are only $O_n(1)$ good choices of $l$, and since there are at most $n$ such possible cases, this contribution can be ignored. Thus, we only consider the cases where $a_{ij}$ and $a_{kj}$ are at corresponding places in a palindrome.

We now consider the case when $k = i + cN2^{-n}$, hence $a_{ij}$ and $a_{kj}$ must be on the same side of the main diagonal in order to match. Moreover, to have $k\in\{1,2,\ldots,N\}$ we must have $i\in\{1,2,\ldots,N - cN2^{-n}\}$. Another restriction arises from the fact that they are on the same side of the main diagonal. We note that we won't cross the main diagonal when moving down from any $a_{ij}$ below the main diagonal to $a_{kj}$. There will similarly be no crossing if $a_{ij}$ lies more than $cN2^{-n}$ elements above the main diagonal. This defines two right-triangular regions of height $N-cN2^{-n} +O_n(1)$, which in total gives a square of area
\be\left(\frac{2^n-c}{2^n}\right)^2 N^2+O_n(N)\ee from which to choose $a_{ij}$, thus giving that many valid choices of $a_{ij}$. We also have the restriction on the values of $l$ as explained in Lemma \ref{thm:properties}, leaving $N - cN2^{-n} + O_n(1)$ good values of $l$ for each of these $a_{ij}$. In total $cN2^{-n}$ contributes
\be\left(\frac{2^n-c}{2^n}\right)^3 N^3+O_n(N^2)\ee matchings to the fourth moment.

Next we consider the case where we cross the main diagonal when moving from $a_{ij}$ to $a_{kj}$, so that $k = i + cN2^{-n} -1$. In this case, the area of values of $a_{ij}$ from which we will cross the diagonal to give a matching will be mostly defined by the parallelogram bordered by the triangles from the previous constant. However, there may also be additional strips as depicted in light shading in Figure \ref{fig:adjareas}, but these will only be of width $1$, so the area is essentially that of the parallelogram  of height $N - cN 2^{-n} + O_n(1)$ and width $cN 2^{-n} + O_n(1)$. There will also be $cN 2^{-n} + O_n(1)$ good values of $l$, so in all this constant contributes
\be\left(\frac{2^n c^2 - c^3}{2^{3n}}\right) N^3 + O_n(N^2)\ee matchings.\end{proof}

\begin{figure}[htb]
\begin{center}
\scalebox{.7}{\includegraphics{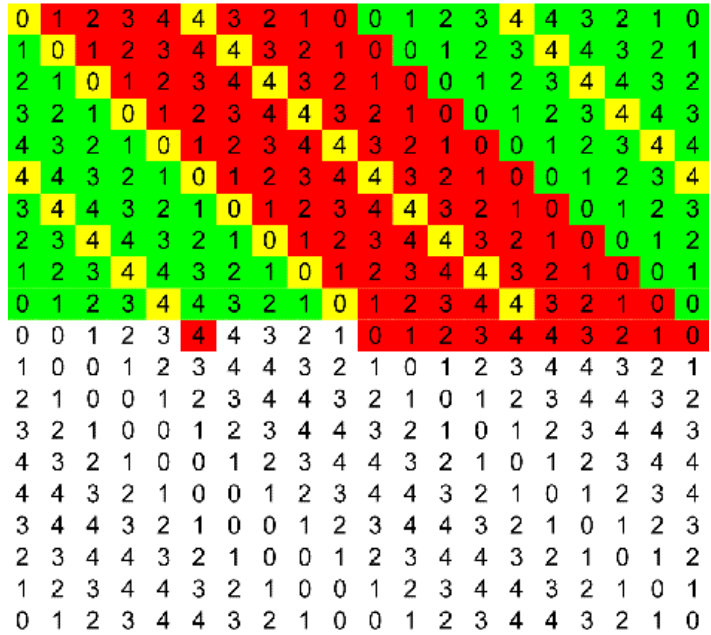}}
\caption{\label{fig:adjareas} Regions where $k = i + \frac{N}{2}$ gives a matching are indicated in medium shading, whereas those where $k = i + \frac{N}{2} - 1$ are indicated in dark shading. Regions where both are satisfied are indicated in light shading: These are $1-$dimensional, and thus won't contribute in the limit.}
\end{center}\end{figure}

We argue similarly for the negative constants $\{-cN2^{-n}, -(cN2^{-n}-1)\}$ for $c\in\{1,2,\ldots, 2^n - 1\}$, in which case we are now moving up $c$ palindromes, and either crossing or not crossing the mian diagonal, respectively. We easily see that this is essentially switching the roles of $a_{ij}$ and $a_{jk}$, so the contributions will be the same. If we repeat the analysis above we find regions of identical size that thus give identical contributions to the fourth moment.  Pictorially, what happens for a negative constant is that of the positive one rotated $180$ degrees. Thus, the contribution to the fourth moment will be given by the contributions from the positive constants ($\{cN2^{-n}, cN2^{-n}-1\}$ for $b\in\{1,2,\ldots, 2^n - 1\}$ multiplied by a factor of $2$ to account for the negative constants.

\begin{thm}\label{lem:contr4thmomentgeneraln} Fix $n$ and consider the limit as $N\to\infty$ of the average fourth moment of our ensemble. The contribution from one of the adjacent matching configurations (i.e., $a_{ij}=a_{jk}$ and $a_{kl}=a_{li}$) to this limit is
\be\label{eq:4thmom}\frac{2}{3}2^n +\frac{1}{3}2^{-n}.\ee Since all configurations contribute the same main term by Lemma \ref{thm:4momVP}, we have
\be\label{eq:4thmomtot} M_{4,n}\ =\ 2^{n+1} + 2^{-n},\ee which is asymptotic to $2^{n+1}$ as $n\to\infty$.
\end{thm}

\begin{proof}  For each value of $c$, we note that the contribution to $M_{4,n}(N)$ is
\begin{align}\frac{1}{N^3}\left(\left(\frac{2^n-c}{2^n}\right)^3 N^3 + \left(\frac{2^nc^2-c^3}{2^{3n}}\right)N^3+O_n(N^2)\right) \nonumber\\
 \qquad = \left(\frac{2^n-c}{2^n}\right)^3 + \frac{2^n c^2-c^3}{2^{3n}} + O_n\left(\frac{1}{N}\right).\end{align}
Taking the limit as $N\to\infty$ yields the contribution to $M_{4,n}$ is
\be\left(\frac{2^n-c}{2^n}\right)^3 + \frac{2^nc^2-c^3}{2^{3n}}.\ee

We sum over all values of $c$, multiply by $2$ to account for the negative constants, and include the contribution from $C=0$, known to be 1 from \cite{MMS} to obtain the contribution to the fourth moment from the adjacent matching case:
\be
M_{4,n}^{{\rm adj}}\ =\   1+\frac{2}{2^{3n}}\sum_{c=1}^{2^n}\left((2^n-c)^3 + 2^n c^2 - c^3\right).\ee Extending the sum above to include $c=0$ cancels the first and last terms of the sum, but we must subtract $4$ to compensate. This then leaves a sum of squares which is easily evaluated:
\begin{align}M_{4,n}^{{\rm adj}}\  &= \ -1+\frac{2}{2^{3n}}\sum_{c=0}^{2^n} 2^n c^2 \nonumber\\
&= \ -1 + \frac{2}{2^{2n}}\frac{2^n(2^n+1)(2\cdot 2^n+1)}{6} \nonumber\\
&= \ -1 + \frac{(1+2^{-n})(2\cdot 2^n+1)}{3} \nonumber\\
&= \ -1 + \frac{1}{3} (2\cdot 2^n + 2 + 1 + 2^{-n}) \nonumber\\
&= \frac{2}{3} 2^n + \frac{1}{3} 2^{-n}.\end{align} Multiplying by three to account for the two adjacent configurations and one non-adjacent configuration then yields
\be M_{4,n}\ =\ 2^{n+1} + 2^{-n},\ee the desired result.\end{proof}

\subsubsection{Adjacent Matching Case for Doubly Palindromic Toeplitz Matrices}

To calculate the contribution from the case of all adjacent matchings to even moments of doubly palindromic Toeplitz matrices, we simply generalize the pictorial method from above. For the $2m\textsuperscript{th}$ moment, we find that our final system of equations for the adjacent configuration becomes
\begin{align*}
i_3 & \ =  \ i_1 + C_1 \\
i_5\ =\ i_3 + C_2 &\ =\ i_1 + C_1 + C_2 \\
&\ \ \vdots \\
i_1\ =\ i_{2m-1} + C_m &\ =\ i_1 + \sum_{k=1}^m C_k. \end{align*}

\begin{rek}\label{indexil} The even indices don't appear because the $n\textsuperscript{th}$ matching gives the equation $i_{2n}-i_{2n-1}\ =\ -(i_{2n+1}-i_{2n}) + C_n$, and the $i_{2n}$ terms cancel.  However, for each non-zero constant $C_l$, we will have a picture similar to Figure \ref{fig:adjrows}, which limits the number of good values of the even indices $i_{2l}$ analogous to the restrictions on $l$ for the fourth moment. Moreover, as each $i_{2n+1}$ is related back to $i_1$, the difference between the maximum and minimum partial sums must be strictly less than $N+O(1)$ or we lose a degree of freedom.\end{rek}

Before deducing the main term in the case of all adjacent matchings, we first set our notation.

\begin{defi} A \textbf{C-vector} is the ordered collection of constants relating the odd indices to each other. In the example at the beginning of this subsection, the C-vector would be $(C_1, C_2, \ldots, C_m)$.
A \textbf{core} of a C-vector is the ordered collection of nonzero constants in the C-vector. That is, we collapse down the C-vector to its core by removing all of the zero constants from it. We can also think of the C-vector as being built up from the core by adding back the zeros in the correct places.
\end{defi}

\begin{thm}\label{thm:dptm}The contribution of one adjacent configuration to the $2m\textsuperscript{th}$ moment averaged over the ensemble of doubly palindromic Toeplitz matrices equals
\be \label{eq:momcal}
-2 + 2^{-m}\left(\sum_{b=1}^{3} b^m\right). \ee
\end{thm}

\begin{proof}
The following observations greatly simplify the analysis for this case:

\bi

\item If the constants $\pm\frac{N}{2}$ appear in the $C$-vector $(C_1, \dots, C_n)$, then $\pm\frac{N}{2}-1$ cannot occur as a main term. If it did, we would lose a degree of freedom in $i_2$, as $a_{i_1i_2}$ would need to lie on a very specific set of diagonals.

\item If some $C_j$ is non-zero, then the next non-zero $C$ chosen must be $-C_j$, as we would otherwise lose a degree of freedom in $i_1$.

\ei

Now consider the $2m\textsuperscript{th}$ moment, which will have a C-vector of length $m$. We can then consider a subset of length $k$ ($k$ even) of $(\frac{N}{2}, -\frac{N}{2},\ldots,\frac{N}{2}, -\frac{N}{2})$ that forms the core of the $C$-vector, with the remaining entries being zero. There are then $\binom{m}{k}$ distinct ways to insert the remaining zeros, and thus $\binom{m}{k}$ ways to build a C-vector around this core.

We now consider the contribution from each of these $C$-vectors. By Remark \ref{indexil}, there will be $\frac{N}{2}$ values of $i_1$ to choose from, and there will be $k$ other $i_{2l}$ (corresponding to the $k$ nonzero $C_l$) that will have $(N-\frac{N}{2})+O(1)$ good values. Thus the contribution for each of these cases will be $(\frac{1}{2})^{k+1}$.  The total contribution to the $2m\textsuperscript{th}$ moment from this configuration, summing over all possible $C$-vectors, is therefore
\be \sum_{\substack{k\text{ even} \\ k=2}}^m \binom{m}{k}\left(\frac{1}{2}\right)^{k+1}. \ee
If we pull out a factor of $\frac{1}{2}$ and include $m=0$ in the sum, we can use the binomial theorem to express this as
\be \label{momcalpart}
\frac{1}{4}\left(\left(1+\frac{1}{2}\right)^m + \left(1-\frac{1}{2}\right)^m\right) - \frac{1}{2}. \ee

The contribution from a core of $(-\frac{N}{2}, \frac{N}{2}, \ldots,-\frac{N}{2}, \frac{N}{2})$ will be the same as that above. The cores of $(\pm(\frac{N}{2}-1), \mp(\frac{N}{2}-1), \ldots,\pm(\frac{N}{2}-1), \mp(\frac{N}{2}-1))$ can be similarly analyzed, and they will also have the same contribution since $N-\frac{N}{2} + O(1) = \frac{N}{2} + O(1)$, so we multiply \eqref{momcalpart} by $4$. We also include the contribution from the $0$-vector, which is $1$ for the adjacent case. Thus, the contribution from each configuration is
\be
-2 + \left(1+\frac{1}{2}\right)^m + 1^m + \left(1-\frac{1}{2}\right)^m\ = \ -2 + 2^{-m}\left(\sum_{b=1}^{3} b^m \right),\ee completing the proof.\end{proof}

\begin{rek} Unfortunately, this method does not readily generalize to matrices with a greater number of palindromes. The fundamental reason is that the observations made at the beginning of Theorem \ref{thm:dptm} no longer hold for these matrices, which then makes it tremendously more difficult to generate all valid $C$-vectors.

To demonstrate these difficulties, we investigated the $6$\ts{th} moment of a matrix with four palindromes. While we can construct $C$-vectors in the same manner as for the doubly palindromic ensemble, we clearly will be missing a substantial number of possible $C$-vectors. For instance, we miss the vector $\left(\frac{N}{2}, \frac{N}{4}, -\frac{3N}{4}\right)$.

In addition to these vectors, we have even more problematic vectors such as $\left(\frac{N}{2}, \frac{N}{4}-1, -\left(\frac{3N}{4}-1\right)\right)$, which turns out to be valid for $a_{ij}$ chosen within a certain parallelogram shaped band of the matrix. These new vectors, in which ``mixing'' is important, are hard to systematically account for, making it quite difficult to determine precisely which $C$-vectors to include and which to exclude. While we could in principle calculate the adjacent contribution to any given moment for any number of palindromes, there is no apparent method that will simultaneously work for all of these possibilities. Similar to other investigations on related ensembles, we are left with general existence proofs of the moments, as well as estimates on their rate of growth. \end{rek}

\subsection{General Even Moments} We extend our analysis to consider the general even moments. We expect Lemma \ref{thm:4momVP} to hold for all general even moments as in the case of single palindromic Toeplitz matrices; in other words, we expect in the limit as $N\to\infty$ each matching configuration to contribute equally. Given a general even $2m\ts{th}$ moment, for each configuration of this moment we have a system of $m$ equations with $2m$ unknown indices. In the case of single palindromic Toeplitz matrices, it is known that we can always choose $m+1$ free indices and the $C$-values such that all $m$ equations are satisfied. Furthermore, because $C$ can only be $\{ 0, \pm (N-1) \}$ in the case of single palindromic Toeplitz matrices (all other choices give a contribution of size $O( \frac{1}{N})$), there is only one valid choice for each of the remaining $m-1$ indices and therefore each configuration contributes 1 to the $2m\ts{th}$ moment. However, in the case of highly palindromic Toeplitz matrices there are more choices for the $m-1$ indices. As the result, the lower order term, which was previously negligible, starts contributing to the moment. Nonetheless we conjecture that the contribution from this new term is the same for all configurations and sketch a detailed analysis in Appendix \ref{sec:configurationconj}. If we could extend Lemma \ref{thm:4momVP}  to the general even $2m\ts{th}$ moment, then we would have greatly reduced the complexity of our moment problem as we  would only need to calculate the contribution of the completely adjacent matching, and immediately get the same contribution from the other $(2m-1)!! - 1$ configurations.

The arguments above, as well as those in the appendices, lead us to make the following conjecture.

\begin{conj}[Configurations Conjecture]\label{conj:configurationconjequality} Let $n$ be a fixed positive integer, and for each $N$ a multiple of $2^n$ consider the ensemble of real symmetric $N\times N$ palindromic Toeplitz matrices whose first row is $2^n$ copies of a fixed palindrome, where the independent entries are independent, identically distributed random variables arising from a probability distribution $p$ with mean 0, variance 1 and finite higher moments. Then as $N\to\infty$ each of the $(2m-1)!!$ configurations for the $2m\ts{th}$ moment contribute the same main term. In particular, this means the general $2m$\textsuperscript{th} moment is just $(2m-1)!!$ times the contribution from the configuration of all adjacent matchings.
\end{conj}

We provide some numerical support for this conjecture in Appendix \ref{sec:numerics} and some theoretical evidence in Appendix \ref{sec:configurationconj}.

\subsection{Moment bounds and fat tails}

We now extend Theorem \ref{thm:dptm} to matrices with greater palindromicity. In doing so, we miss many of the $C$-vectors that contribute to these moments, but exact calculations for even a quadruply palindromic matrix have proven difficult. The goal is to obtain good enough bounds on the moments to deduce properties of the limiting spectral measures. We begin with the following lemma.

\begin{lem}\label{lem:boundevenmomentsnlower} Fix $n$ and consider $M_{2m,n}$, the limit as $N\to\infty$ of the average of the $2m$\textsuperscript{th} moment of our ensemble (the set of $N\times N$ real symmetric palindromic Toeplitz matrices where the first row contains $n$ palindromes). Then for $m$ sufficiently large, $M_{2m,n} > (2m-1)!!$, or in other words $M_{2m,n}$ is larger than the corresponding moment of the standard normal. If we assume Conjecture \ref{conj:configurationconjequality} then we may improve the lower bound to \be M_{2m,n} \ \ge \ \left(-2\cdot (2^n-1) + 2^{-mn}\left(\sum_{c=1}^{2\cdot2^n-1} c^m\right)\right) \cdot (2m-1)!!. \ee As $m\to\infty$, we have \be M_{2m,n} \ \gg \  \frac{2^{m+n}}{m} \cdot (2m-1)!!.\ee
\end{lem}

\begin{proof}  Let $C_c=\frac{cN}{2^n}$ for $b\in\{1,\ldots,2^n-1\}$. We determine the contribution to the average $2m$\ts{th} moment as $N\to\infty$ from one of the two adjacent matchings. That is, consider the core (i.e., the non-zero part) of the corresponding $C$-vectors is $(\pm C_c, \mp C_c,\dots, \pm C_c, \mp C_c)$ and its complement $(\pm(N-1-C_c), \mp(N-1-C_c),\ldots, \pm(N-1-C_c), \mp(N-1-C_c))$. They contribute
\be -2 + \left(2 - \frac{c}{2^n}\right)^m + \left(\frac{c}{2^n}\right)^m.\ee The proof of this claim goes back to the observation in Figure \ref{fig:adjrows} that for $C_c=\frac{cN}{2^n}$, the number of free $l$ values is $N-\frac{cN}{2^n} + O(1)$, whereas if $C_c=\frac{cN}{2^n}-1$, then there are $\frac{cN}{2^n} + O(1)$ good $l$ values.  Thus the complementary $C_c$ will give the same restrictions on the number of $l$ values.

Moreover, the restrictions on $i_1$ from $\frac{cN}{2^n}$ and $N-1-\frac{cN}{2^n}$ sum to $1$. Thus, as there are the two cases (plus first or minus first) for each, when we sum them and extend the sums back to $0$, we have
\be
-2 + 2\cdot \sum_{j\ {\rm even}}^m \binom{m}{j}\left(\frac{2^n-c}{2^n}\right)^{j+1} = -2 + \left(\frac{2\cdot 2^n -c}{2^n}\right)^m + \left(\frac{c}{2^n}\right)^m.\ee

In order to get our lower bound for $M_{2m,n}$, we repeat this for every value of $c\in\{1,2,\ldots,2^n -1\}$.  Adding in the contribution from the zero vector, we obtain
\be
-2\cdot (2^n-1) + 2^{-mn}\left(\sum_{c=1}^{2\cdot2^n-1} c^m\right),\ee which is easily summed for any value of $k$.

If we assume Conjecture \ref{conj:configurationconjequality}, each of the $(2m-1)!!$ matchings contribute equally, and hence \be M_{2m,n}\ \geq \ \left(-2\cdot (2^n-1) + 2^{-mn}\left(\sum_{c=1}^{2\cdot2^n-1} c^m\right)\right) \cdot (2m-1)!!. \ee The behavior for large $m$ follows by approximating the sum with an integral.

If we do not assume Conjecture \ref{conj:configurationconjequality}, we instead use the fact that each configuration contributes at least 1 (this follows from the analysis of the analysis of the single palindromic case of \cite{MMS}) and the analysis above for the case of all adjacent matchings. This yields \be M_{2m,n} \ \gg\ \frac{2^{m+n}}{m} + (2m-1)!!-1,\ee which exceeds $(2m-1)!!$ for $m$ sufficiently large.\end{proof}

\begin{rek} We can slightly improve the lower bound in the case when we do not assume  Conjecture \ref{conj:configurationconjequality} by noting that there are two cases where we have all adjacent matchings. Of course, this improvement is negligible in the limit, and it is only under the assumption of Conjecture \ref{conj:configurationconjequality} that we can significantly improve the lower bound.
\end{rek}

We can now turn to an analysis of the properties of the limiting spectral measures. Note $n=0$ corresponds to the real symmetric palindromic Toeplitz matrices studied in \cite{BDJ,MMS}, and $n=1$ corresponds to the doubly palindromic Toeplitz matrices studied in \cite{MMS}. We now prove Theorem \ref{thm:fattails}, which we restate below for the reader's convenience.\\ \

\noindent \textbf{Theorem \ref{thm:fattails}\ (Fat Tails).} \emph{Consider the ensemble from Theorem \ref{thm:mainweakstrong}. For any fixed $n \ge 2$, the moments grow faster than the corresponding moments of the standard normal. If we additionally assume Conjecture \ref{conj:configurationconjequality}, then letting $M_{2m,n}$ denote the $2m$\textsuperscript{th} moment of the limiting spectral measure of our ensemble for a given $n$, we have \be M_{2m,n} \ \gg \  \frac{2^{m+n}}{m} \cdot (2m-1)!!.\ee  The limiting spectral measure thus has unbounded support, and fatter tails than the standard normal (or in fact any of the known limiting spectral measures arising from an ensemble where the independent entries are chosen from a density whose moment generating function converges in a neighborhood of the origin).}

\begin{proof} From Lemma \ref{lem:boundevenmomentsnlower} we know $M_{2m,n} > (2m-1)!!$ for $m$ large, and under the assumption that Conjecture \ref{conj:configurationconjequality} holds we additionally have \be M_{2m,n} \ \gg \  \frac{2^{m+n}}{m} \cdot (2m-1)!!.\ee As $n \ge 1$, for $m$ large this is greater than the $m$\ts{th} moment of the standard normal, which is $(2m-1)!!$. Thus our limiting spectral measure has unbounded support, and more mass in the tails than the standard normal, or in fact, any normal if $n \ge 2$. To see the last claim, note that if $X \sim N(0,\sigma^2)$ then the $2m$\textsuperscript{th} moment of $X$ is $\sigma^{2m} \cdot (2m-1)!!$, and thus when $n \ge 2$ eventually the moment of our ensemble is greater than the moment of this normal.
\end{proof}













\appendix

\section{Numerical Methods}\label{sec:numerics}

While they can never be accepted as proof\footnote{When investigating the case of real symmetric palindromic Toeplitz matrices years ago, both the authors of \cite{BDJ} and \cite{HM} looked at the numerical simulations and thought the results looked Gaussian; it was only after a more careful analysis that the small deviations were isolated.}, numerical simulations did much to guide our efforts in attacking this problem, and we would not have been successful without it, as our naive adaptations of previous works on this subject failed to give even remotely accurate predictions. Therefore we give a brief outline of our use of these simulations below.

Initially we primarily used a basic, direct method to approximate the moments of the eigenvalue distribution. We first set up a matrix with $2^n$ palindromes and chose $N$ so that the matrix had the desired form (every element appears exactly $2^{n+1}$ times in the first row). For each moment we used the eigenvalue trace lemma to calculate the moment of the eigenvalue distribution for this particular matrix, then we averaged over a large number of such random matrices to get an approximation for that moment averaged over the ensemble of Toeplitz matrices with $2^n$ palindromes. To get increased accuracy, we simply increased $N$.

Our calculations were successful in verifying our predictions for the higher moments of the doubly palindromic Toeplitz matrix and for the fourth moment of the 64-palindrome Toeplitz matrix, supporting our conjectures. Specifically, whenever one has involved combinatorial arguments such as the ones above, it is worthwhile to numerically test the theory. In Table 1 we present the data from simulating 1000 real symmetric doubly palindromic 2048 $\times$ 2048 Toeplitz matrices, and compare the even moments to our predicted values (as expected, the odd moments were small). Unfortunately the rate of convergence is slow in $N$ due to the presence of large big-Oh constants.

\begin{table}\label{tab:bigsim}
\caption{Conjectured and observed moments for 1000  real symmetric doubly palindromic 2048 $\times$ 2048 Toeplitz matrices. The conjectured values come from assuming Conjecture \ref{conj:configurationconjequality}.}
\begin{center}
\begin{tabular}{|r|r|r|r|}
\hline
Moment & Conjectured & Observed & Observed/Predicted\\
\hline \hline
  2 & 1.000 & 1.001 & 1.001 \\
  4 & 4.500 & 4.521 & 1.005 \\
  6 &37.500  & 37.887 & 1.010 \\
  8 &433.125  & 468.53\ \ \ & 1.082   \\
  10 & 6260.63\ \ \  & 107717.3\ \ \ \ \ & 17.206 \\
  \hline
\end{tabular}
\end{center}
\end{table}

It is worth noting how slow the convergence is. For example, when we considered 1000 real symmetric doubly palindromic $96 \times 96$ Toeplitz matrices, the observed second moment was 0.990765, the fourth moment was 4.75209, the sixth was 45.7965 (for a ratio of 1.22) and the eighth was 737.71 (for a ratio of 1.70).

While this method was quite useful and accurate for lower moments or for a small number of palindromes, for larger values of these quantities the big-Oh constants grew quite large, making it computationally prohibitive to simulate a representative sample of sufficiently large matrices, and thus leaving us with a rather poor estimate of the moments and providing no guide to whether or not our formulas were accurate.

To avoid simulating ever-larger matrices, we instead realized that the average $2m$\ts{th} moment of $N$ by $N$ matrices in our ensembles should satisfy
\be M_{2m,n;N}\ =\ M_{2m,n} + \frac{C_1}{N} + \frac{C_2}{N^2} + \cdots + \frac{C_m}{N^m}.\ee Thus, rather than simulating prohibitively large matrices, we could instead simulate large numbers (to increase the likelihood of a representative sample) of several sizes of smaller matrices and then perform a least squares analysis to estimate the value of $M_{2m,n}$. We present the data in Table 2.

\begin{table}\label{tab:obsmomentsvarsizes}
\begin{center}
\caption{Observed moments for the ensembles of $N\times N$ doubly palindromic (i.e., $n=2$) Toeplitz matrices. The conjectured values are obtained by assuming Conjecture \ref{conj:configurationconjequality}, which means taking the contribution from the adjacent matching case and multiplying by $(2m-1)!!$ for the $2m$\textsuperscript{th} moment.}
\begin{tabular}{|r|r|r|r|r|r|r|}
  \hline
N	&	sims	&	2nd	&	4th	&	6th &	8th 	&	10th	\\ \hline \hline
8	&	1,000,000	&	1.000	&	8.583	&	150.246	&	3984.36	&	141270.00	\\
12	&	1,000,000	&	1.000	&	7.178	&	110.847	&	2709.61	&	90816.60	\\
16	&	1,000,000	&	1.001	&	6.529	&	93.311	&	2195.78	&	73780.00	\\
20	&	1,000,000	&	1.001	&	6.090	&	80.892	&	1790.39	&	57062.50	\\
24	&	1,000,000	&	1.000	&	5.818	&	73.741	&	1577.42	&	49221.50	\\
28	&	1,000,000	&	1.000	&	5.621	&	68.040	&	1396.50	&	42619.90	\\
64	&	250,000	&	1.001	&	4.992	&	50.719	&	858.58	&	22012.90	\\
68	&	250,000	&	1.000	&	4.955	&	49.813	&	831.66	&	20949.60	\\
72	&	250,000	&	1.000	&	4.933	&	49.168	&	811.50	&	20221.20	\\
76	&	250,000	&	1.000	&	4.903	&	48.474	&	794.10	&	19924.10	\\
80	&	250,000	&	1.000	&	4.888	&	47.951	&	773.31	&	18817.00	\\
84	&	250,000	&	1.001	&	4.876	&	47.615	&	764.84	&	18548.00	\\
128	&	125,000	&	1.000	&	4.745	&	44.155	&	659.00	&	14570.60	\\
132	&	125,000	&	1.000	&	4.739	&	43.901	&	651.18	&	14325.30	\\
136	&	125,000	&	0.999	&	4.718	&	43.456	&	637.70	&	13788.10	\\
140	&	125,000	&	1.000	&	4.718	&	43.320	&	638.74	&	14440.40	\\
144	&	125,000	&	1.001	&	4.727	&	43.674	&	647.05	&	14221.80	\\
148	&	125,000	&	1.000	&	4.716	&	43.172	&	628.02	&	13648.10	\\
  \hline \hline
Conjectured & & 1.000 & 4.500	&	37.500	&	433.125	 & 6260.63 \\ \hline
Best Fit $M_{2m,2}$ & & 1.000 & 4.496 & 38.186 & 490.334 & 6120.94 \\ \hline
\end{tabular}
\end{center}
\end{table}

In performing the curve fitting, we frequently found big-Oh constants so large that it would have been impossible to sample a sufficiently large sample of matrices to get an accuracy of within a few percent for the moments. For example, for the fourth moment of a Toeplitz matrix with 64 palindromes we found the big-Oh constant to be above $30,000$, implying that averages of quite large matrices would give an approximation for the fourth moment that would be off by 10, compared to a true value of about 128.

For the doubly palindromic Toeplitz matrices in Table 2, we see that our best fit constants for $M_{2m,2}$ with $2m \le 10$ are quite close to the values predicted by Conjecture \ref{conj:configurationconjequality}, providing strong evidence supporting it.





\section{Conjectures for General Configurations}\label{sec:configurationconj}

In Conjecture \ref{conj:configurationconjequality} we hypothesize that the main term of the contribution of any configuration is equal in the limit. We provide some arguments in support of this.


To extend Lemma \ref{thm:4momVP} to the general even moment, we introduce some notation for a ``lift map'', which is a way of relating one configuration of an even moment (say one of the $(2m-1)!!$ configurations of the $2m$\textsuperscript{th} moment) to one configuration of the next higher even moment (to one of the $(2m+1)!!$ configurations of the $(2m+2)$\textsuperscript{nd} moment. If we add a pair of entries to a configuration, this moves us from our initial configuration to some configuration of the next even moment. There are only two ways to add these entries: adding a pair of adjacent entries or adding a pair of non-adjacent entries.

\begin{conj}[Configuration Lifting - Adjacent Case]\label{thm:lift1} Consider a configuration of the $2k\ts{th}$ moment. All configurations of the $(2k+2)\ts{nd}$ moment obtained by adding a pair of adjacent entries to this configuration contribute equally to the  $(2k+2)\ts{nd}$ moment.
\end{conj}

\begin{figure}[htb]
\begin{center}
\scalebox{.8}{\includegraphics{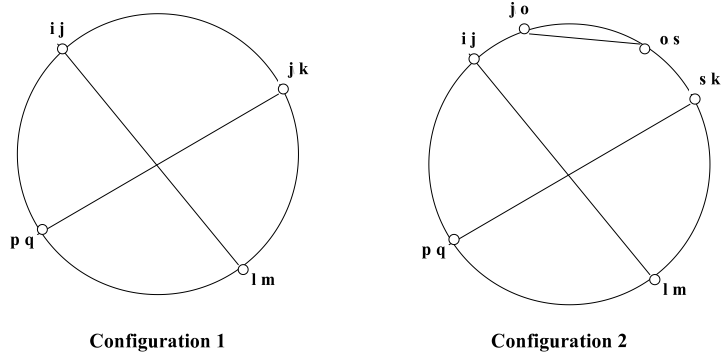}}
\caption{\label{fig:2} Moment Lifting by adding a pair of adjacent entries}
\end{center}\end{figure}

We comment on the above conjecture; see Figure \ref{fig:2} for an example.
Let $$(\dots, a_{p q}, \dots, a_{i j}, a_{j k }, \dots a_{lm}, \dots )$$ be a tuple from one of the $(2k-1)!!$ configurations of the $2k$\ts{th} moment; for brevity's sake we call this configuration (1). We let $$(\dots, a_{pq}, \dots, a_{ij}, a_{jo}, a_{os}, a_{sk }, \dots, a_{lm}, \dots)$$ be the new tuple obtained by adding the pair of adjacent, matched entries $a_{jo} = a_{os}$; we denote this by configuration (2). Let $\Omega_{2k}$ be the set of all tuples that work for configuration (1) and $\Omega_{2k+2}$ be the set of all tuples that work for  configuration (2). We define a ``lift map'' $F: \Omega_{2k} \to \Omega_{2k+2}$ by
\be F\left( (\dots, a_{p q}, \dots, a_{i j}, a_{j k }, \dots, a_{lm}, \dots )\right) \ = \ (\dots, a_{pq}, \dots, a_{ij}, a_{jo}, a_{os}, a_{sk }, \dots, a_{lm}, \dots ).\ee

Note $F$ maps each index in $ (\dots, a_{p q}, \dots, a_{i j}, a_{j k }, \dots, a_{lm}, \dots )$ to itself and inserts a new index $s = j - B + B^{'}$ where $B$ is the value of $C$ corresponding to the pair of entries ($a_{jk} = a_{pq}$), and $B^{'}$ is any value of $C$ such that $s \in \{1,\dots,N\}$ and $(B - B^{'})$ is a valid value of $C$. The  system of equations corresponding to the tuples $(\dots, a_{p q}, \dots, a_{i j}, a_{j k }, \dots, a_{lm}, \dots )$ is given as follows:
\be
\begin{cases}
l -m \ = \ - ( i - j ) + A \\
p-q \ = \ - ( j - k) + B \\
\dots \\
\dots. \\
\end{cases}
\ee

Under the map $F$, we obtain a new tuple $(\dots, a_{pq}, \dots, a_{ij}, a_{jo}, a_{os}, a_{sk }, \dots, a_{lm}, \dots )$ satisfying the system of equations
\be
\begin{cases}
l -m \ = \ - ( i - j ) + A \\
p-q \ = \ - ( s - k) + B^{'} \\
j - o \ = \ - (o - s) + (B - B^{'}) \\
\dots. \\
\end{cases}
\ee

Except for the two equations $p-q = - ( s - k) + B^{'}$  and $j - o = - (o - s) + (B - B^{'})$, every other equation of configuration (1) is preserved under $F$ and therefore still holds in configuration (2). Furthermore, since both $B^{'}$ and $(B - B^{'})$ are valid choices of the $C$ value by  the construction of $F$, the two equations $p-q = - ( s - k) + B^{'}$  and $j - o = - (o - s) + (B - B^{'})$ are also valid. Thus
\be\label{eq:newtup} (\dots, a_{pq}, \dots, a_{ij}, a_{jo}, a_{os}, a_{sk }, \dots, a_{lm}, \dots ) \in \Omega_{2k+2}. \ee

Based on the analysis above, the mapping $F$ only depends on the choice of one index $j$ and one $C$ value $B$ from the original configuration (1). First, we can take $B$ to be any possible value of our $C$ values since in order to obtain all the tuples of configuration (1), we need to sum over all possible combinations of the $C$'s that work for the system of equations corresponding to configuration (1).

Second, for any configuration in the $2k$\textsuperscript{th} moment, we have $2k$ indices (unknown variables) and $k$ equations with the last equation linearly dependent on the rest. Thus we must have at least two completely free indices that can take on any value between $1$ and $N$. We specify the two completely free indices by choosing the very first entry of the tuple, which obviously can be any vertex of the configuration. Starting with any vertex, if there are some $m$ satisfying tuples with $a_{ij}$ being our first entry, then there must be exactly $m$ tuples with $a_{j i}$ being our first entry since the matrix is symmetric over the main diagonal. Therefore, the number of resulting tuples from the map $F$ is the same regardless of where we add the adjacent pair.

The biggest obstacle left is whether the map $F$ can reach every possible tuple of configuration (2). Given a tuple $(\dots, a_{pq}, \dots, a_{ij}, a_{jo}, a_{os}, a_{sk }, \dots, a_{lm}, \dots )$ of the configuration (2), the inverse map $F^{-1}$ simply substitutes the equation arise from the added pair of adjacent entries ($a_{jo} = j_{os}$) $j = s + C_3$ into the equation arise from the entries ($a_{sk} = j_{pq}$) $p - q = - (s - k) + C_1$. In order for this tuple to be reachable by the map $F$, it suffices that $C_2 + C_3$ is a valid $C$ value. It is not a problem for the $4^{th}$ and $6^{th}$ moment since the sum of all the $C$ values corresponding to any tuple must equal 0. However, at higher even moment, there is possibility that $C_2 + C_3$ is not a valid $C$ value. So there might be tuples of the configuration (2) that the map $F$ can not reach. If those tuples only contribute to some lower order term then all configurations of the $(2k+2)\ts{nd}$ moment obtained by adding a pair of adjacent entries to a given configuration of the $(2k)\ts{th}$ moment would contribute equally.

\begin{conj}[Configuration Lifting - Nonadjacent Case]\label{thm:lift2} Consider a configuration of matchings for the $2k\ts{th}$ moment. All configurations at the $2k\ts{th}$ moment obtained by adding a pair of non-adjacent entries contribute equally to the  $(2k+2)$\textsuperscript{nd} moment.\end{conj}

\begin{figure}[htb]
\begin{center}
\scalebox{.8}{\includegraphics{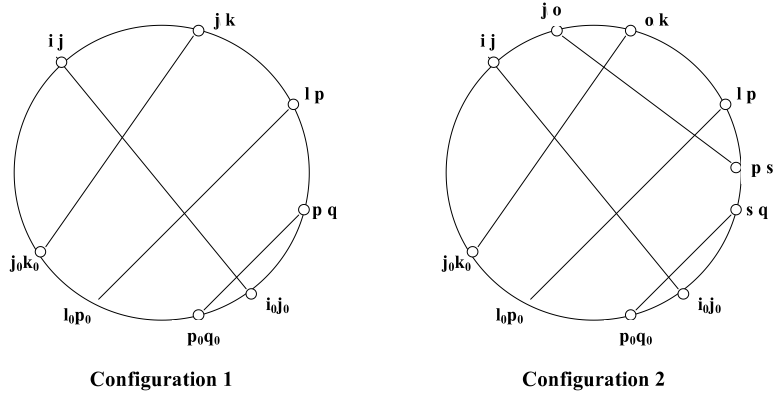}}
\caption{\label{fig:3} Moment Lifting by adding a pair of non-adjacent entries}
\end{center}\end{figure}

We provide some arguments in support of our claim; see Figure \ref{fig:3} for an example. Let $(\dots, a_{ij}, a_{jk}, \dots, a_{lp}, a_{pq }, \dots )$ be a tuple of a configuration for the $2k\ts{th}$ moment; denote this by configuration (1). We let $(\dots, a_{ij}, a_{jo}, a_{ok}, \dots, a_{lp}, a_{ps}, a_{sq}, \dots )$ be the new tuple obtained by adding the pair of entries $a_{jo} = a_{ps}$. As before, let $\Omega^{'}_{2k}$ be the set of all tuples that work for configuration (1) and $\Omega^{'}_{2k+2}$ be the set of all tuples that work for the configuration $(2)$.  We define the ``lift map'' $FF^{'}: \Omega^{'}_{2k} \to \Omega^{'}_{2k+2}$ in this case by
\be F^{'}\left((\dots, a_{ij}, a_{jk}, \dots, a_{lp}, a_{pq }, \dots )\right) \ = \ (\dots, a_{ij}, a_{jo}, a_{ok}, \dots a_{lp}, a_{ps}, a_{sq}, \dots )\ee
such that $F^{'}$ maps every index in $(\dots, a_{ij}, a_{jm}, \dots, a_{lp}, a_{pq }, \dots)$ to itself and adds two new indices $o = j + B - B^{'}$ and $s = p + D - D^{'}$ where $B$ and $D$ are the $C$ values associated with the pairs containing $a_{jm}$ and $a_{pq}$ respectively. Also, $B^{'}$ and $D^{'}$ are any value of $C$ such that $o,s \in \{1,\dots,N\}$ and $( D^{'} + B^{'} - D -  B)$ is a value of $C$.

For the tuple $(\dots, a_{ij}, a_{jm}, \dots, a_{lp}, a_{pq }, \dots )$ we have the following system of equations:
\be
\begin{cases}
i_0  - j_0 \ = \ - ( i - j ) + A \\
j_0  - k_0 \ = \ - (j - k) + B \\
l_0  - p_0 \ = \ - (l - p) + C \\
p_0  - q_0 \ = \ - (p - q) + D \\
\dots \\
\end{cases}
\ee

Under the map $F^{'}$, we obtain a new tuples $(\dots, a_{ij}, a_{jo}, a_{ok}, \dots, a_{lp}, a_{ps}, a_{sq}, \dots )$ satisfying the system of equations:
\be
\begin{cases}
i_0  - j_0 \ = \ - ( i - j ) + A \\
j_0  - k_0 \ = \ - (o - k) + B^{'} \\
l_0  - p_0 \ = \ - (l - p) + C \\
p_0  - q_0 \ = \ - (s - q) + D^{'} \\
j-o \ = \ - (p -s) + ( D^{'} + B^{'} - D -  B)\\
\dots \\
\end{cases}
\ee

Similar to the analysis in Conjecture \ref{thm:lift1}, all equations except for $j_0  - k_0 = - (o - k) + B^{'}$ and $p_0  - q_0 = - (s - q) + D^{'} $ and $j-o = - (p -s) + ( D^{'} + B^{'} - D -  B)$ are preserved under the map $F^{'}$ so they still hold. Furthermore, since $B^{'}$, $D^{'}$ and  $(D^{'} + B^{'} - D -  B)$ are all valid choices of $C$, the other three equations also hold true. Lastly, the existence of at least two completely free indices allow us to choose them to be the first index in $a_{jm}$ and the second index in $a_{lp}$. Thus, following the same line of argument in Conjecture \ref{thm:lift1}, we can always choose the two free indices such that the number of tuples resulting from the map $F^{'} (\Omega^{'}_{2k})$ are the same regardless of where we add the non-adjacent pair.

\begin{cor}\label{thm:replacement} Given any configuration, we can replace one of its adjacent pairs by another adjacent pair, and similarly for non-adjacent pairs, without changing its contribution to the corresponding moment.\end{cor}

Given any configuration at the $(2k+2)$\textsuperscript{nd} moment, Corollary \ref{thm:replacement} allows us to repeatedly replace adjacent pairs with other adjacent pairs, and similarly for non-adjacent pairs. By iterating this process, we can move any configuration down to the completely adjacent configuration (which contains only adjacent pairs of entries); see for instance Figure \ref{fig:nonadj}. Given any configuration, whenever there is a crossing between two pairs of entries, which implies those entries are non-adjacent, we can eliminate the crossing by replacing one of these pairs of entries with another pair of non-adjacent entries.

\begin{figure}[h]
\begin{center}
\scalebox{.8}{\includegraphics{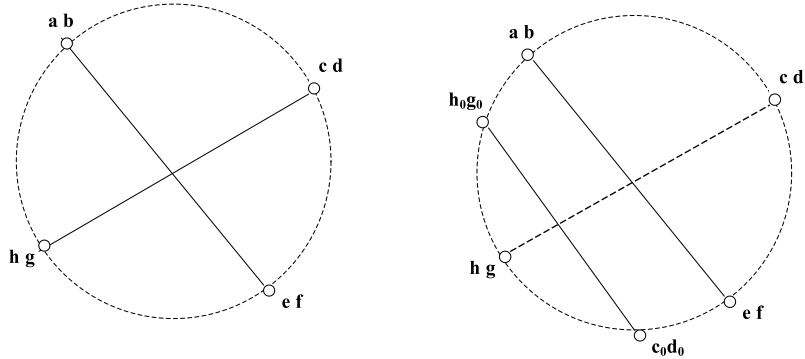}}
\caption{\label{fig:nonadj} Eliminating crossings in a given configuration}
\end{center}\end{figure}

Iterating this process, we end up with a configuration that has no crossing between any of its pairs of entries. We then keep replacing the pair of adjacent entries at the two ends of the configuration with another pair of adjacent entries in the middle of the configuration until we obtain the completely adjacent configuration; see Figure \ref{fig:adj}.

\begin{figure}[h]
\begin{center}
\scalebox{.8}{\includegraphics{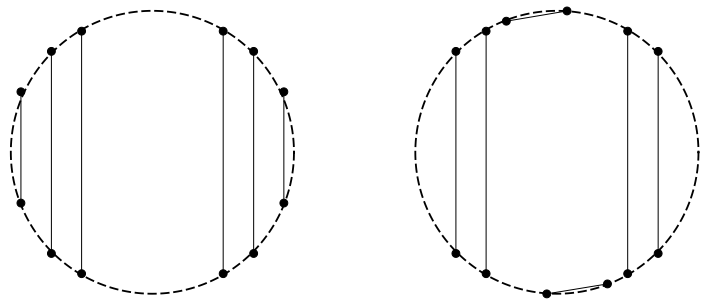}}
\caption{\label{fig:adj} Transformation to completely adjacent configuration}
\end{center}\end{figure}

Based on Corollary \ref{thm:replacement} and Conjecture \ref{thm:4momVP}, it follows by induction that every configuration at any even moment would contribute equally.


\ \\
\end{document}